\documentclass[11pt]{amsart}

\usepackage[dvipsnames]{xcolor}
\usepackage{amssymb,verbatim}
\usepackage{amsmath,amsfonts,enumitem}
\usepackage[mathscr]{euscript} 
\usepackage{amsthm,bm}
\usepackage{url}
\usepackage{graphicx} 
\usepackage{enumitem} 
\usepackage{hyperref} 
\hypersetup{colorlinks} 
\usepackage{bm} 
\usepackage{mathrsfs} 
\usepackage{cancel} 

\usepackage[colorinlistoftodos]{todonotes}

\RequirePackage{cleveref}
\usepackage{hypcap}
\hypersetup{colorlinks=true, citecolor=darkblue, linkcolor=darkblue}
\definecolor{darkblue}{rgb}{0.0,0,0.7}
\newcommand{\darkblue}{\color{darkblue}}

\definecolor{darkred}{rgb}{0.68,0,0}
\newcommand{\darkred}{\color{darkred}}

\definecolor{darkgreen}{rgb}{0,.38,0}
\newcommand{\darkgreen}{\color{darkgreen}}

\newcommand{\defn}[1]{\emph{\darkblue #1}}
\newcommand{\defna}[1]{\emph{\darkred #1}}
\newcommand{\defnb}[1]{\emph{\darkblue #1}}
\newcommand{\defng}[1]{\emph{\darkgreen #1}}

\makeatletter
\def\th@plain{%
	\thm@notefont{}
	\itshape 
}
\def\th@definition{%
	\thm@notefont{}
	\normalfont 
}
\makeatother

\newtheorem{thm}{Theorem}[section]
\newtheorem{lemma}[thm]{Lemma}
\newtheorem{claim}[thm]{Claim}
\newtheorem*{claim*}{Claim}
\newtheorem{cor}[thm]{Corollary}
\newtheorem{prop}[thm]{Proposition}
\newtheorem{conj}[thm]{Conjecture}
\newtheorem{conv}[thm]{Convention}

\theoremstyle{definition}

\newtheorem{rem}[thm]{Remark}

\numberwithin{figure}{section}
\numberwithin{equation}{section}

\newcommand{\BFT}{C_{\mathrm{BFT}}}
\def\zz{\mathbb Z}

\def\de{\delta}
\def\ep{\varepsilon}
\def\eps{\epsilon}

\def\ve{\varepsilon}

\def\cE{\mathcal E}

\def\cF{\mathcal F}

\def\cL{f}

\def\cO{\mathcal O}

\def\cQ{Q}

\def\Eb{\mathbf{E}}

\def\<{\langle}
\def\>{\rangle}

\def\0{{\mathbf 0}}

\def\.{\hskip.06cm}
\def\ts{\hskip.03cm}

\def\.{\hskip.06cm}
\def\ts{\hskip.03cm}

\def\nin{\noindent}

\def\Pb{{\text{\bf P}}}

\DeclareMathOperator{\gap}{\operatorname{gap}} 

\DeclareMathOperator{\dev}{\textnormal{dev}} 

\title{Breaking the Infinite Barrier in the $\frac{1}{3}$--$\frac{2}{3}$ Conjecture}

\newcounter{reviewissue}

\begin{document}
	
	\author[\ts Max Aires]{Max Aires}
	\address[Max Aires]{Department of Mathematics, USC,  Los Angeles, CA 90089.}
	\email{\texttt{aires@usc.edu}}

	\author[Swee Hong Chan]{Swee Hong Chan}
	\address[Swee Hong Chan]{Department of Mathematics, Rutgers University,  Piscataway, NJ 08854.}
	\email{\texttt{sweehong.chan@rutgers.edu}}
	
	\author[Igor Pak]{Igor Pak}
	\address[Igor Pak]{Department of Mathematics, UCLA,  Los Angeles, CA 90095.}
	\email{\texttt{pak@math.ucla.edu}}
	
	\author[\ts Greta Panova]{Greta Panova}
	\address[Greta Panova]{Department of Mathematics, USC,  Los Angeles, CA 90089.}
	\email{\texttt{gpanova@usc.edu}}
	
	\date{\today}
	
	
	
	
	\begin{abstract}
		The balance constant is a poset parameter measuring how well random linear extensions
		can be split according to their relative order on two elements.
		We give an $\ve$-improvement for the balance constant over the bound by
		Brightwell, Felsner and Trotter \cite{BFT95}, for some small $\ve>0$.
		This is the first general result towards the $\frac{1}{3}$--$\frac{2}{3}$ \.
		conjecture in over 30 years.
	\end{abstract}
	
	\maketitle
	
	
	\section{Introduction}\label{s:proof-crit-simp}
	
	When a major open problem remains unresolved for decades, one can certainly
	ask why that is.  In most cases, there are too many plausible answers.
	But for the \. \defnb{$\frac{1}{3}$--$\frac{2}{3}$ \. conjecture},
	there is a standard answer that has also been known for decades: while the conjecture
	is stated for finite posets, it is \emph{false} for \emph{infinite} posets,
	and there are no known tools which make a distinction.  This \defna{infinite barrier}
	has plagued the area for over 30 years, repelling numerous {attempts at
		solving} the conjecture.  In this paper we break the infinite barrier
	by a microscopic margin.
	
	\smallskip
	
	\subsection{Background}\label{ss:intro-back}
	Let \. $P:=(X,\prec)$ \. be a finite poset, where $X$ is a finite set of size $n:=|X|$, and
	\. $\prec$ \. is a { partial order} on~$X$. We write $[n]:=\{1,\ldots,n\}$.
	A \defnb{linear extension} of $P$  is a bijection $f:X \to [n]$ that is order preserving:
	\. $x \prec y$ \. in ${{P}}$ implies $f(x)<f(y)$ for all $x,y \in X$.
	Let $\cE(P)$ denote the set of linear extensions of~$P$.
	
	\begin{conj}[{\rm \defnb{$\frac{1}{3}$--$\frac{2}{3}$ \. conjecture}}{}] \label{conj:1323}
		Let \. $P:=(X,\prec)$ \. be a finite poset that is not a chain.  Then there exist
		elements \ts $x,y\in X$, such that
		\begin{equation}\label{eq:conj-1323}
			\frac{1}{3}  \ \leq \ \Pb[ f(x) \.  < \. f(y)]  \ \leq \ \frac{2}{3}\,,
		\end{equation}
		where the probability is over uniform random \ts $f\in \cE(P)$.
	\end{conj}
	
	The conjecture is a central open problem in order theory, independently proposed by
	Kislitsyn \cite{Kis68} and Fredman \cite{Fre75}, see also \cite{Lin84},
	in connection to sorting with partial information.
	%
	The simplicity and elegance of the \. $\frac{1}{3}$--$\frac{2}{3}$ \. conjecture
	{have attracted much attention}, and numerous special cases of the conjecture have
	been established in the literature.
	
	These known special cases include: width two posets \cite{Lin84},
	posets with symmetries \cite{GHP87}, semiorders \cite{Bri89}, 
	posets of height at most $c\log\log n$
	for a suitable absolute constant $c>0$~\cite{Fri93}, $6$-thin posets \cite{Pec08},
	$N$-free posets \cite{Zag12}, Young diagrams \cite{OS18}, 
	posets whose cover graph is a forest \cite{Zag19}, and posets with at most 14 elements~\cite{Gup26}.  We refer to \cite[$\S$13]{CP-LE}
	{for an introduction to the problem}, additional special cases and further references.
	
	In a major breakthrough, Kahn and Saks~\cite{KS84} used geometric inequalities to
	show that the inequality
	\eqref{eq:conj-1323} holds with slightly weaker constants:
	\begin{equation}\label{eq:KS}
		\frac{3}{11}  \ \leq \ \Pb[ f(x) \.  < \. f(y)]  \ \leq \ \frac{8}{11}\,.
	\end{equation}
	Note that \. $\frac13$--$\frac23$ \. is tight for {the $3$-element poset consisting of a $2$-element chain and an isolated element}, so
	the bound \. $\frac{3}{11} \approx 0.2727<\frac13$ \. leaves relatively little room for
	improvement.  Despite various attempts and simplifications, see e.g.\ \cite{KL90} and
	\cite[$\S$12.3]{Mat02}, the \emph{Kahn--Saks bound} stood unimproved for a decade.
	
	In another major breakthrough, Brightwell, Felsner and Trotter \cite{BFT95} used both
	the combinatorial \emph{Ahlswede--Daykin inequality}  \cite{AD78} and
	the geometric \emph{Alexandrov--Fenchel inequalities}~\cite{BZ-book} to improve on the Kahn--Saks bound:
	\begin{equation}\label{eq:BFT}
		\frac{5-\sqrt{5}}{10}  \ \leq \ \Pb[ f(x) \.  < \. f(y)]  \ \leq \    \frac{5+\sqrt{5}}{10}\,.
	\end{equation}
	Note that \. $\frac{5-\sqrt{5}}{10} \approx 0.2764$, so this might seem an incremental
	improvement which could perhaps be improved as new poset inequalities are added to the
	mix.  {Yet this} \defnb{BFT bound} {remained the best known general bound prior to the present work}.
	
	The reason the BFT bound remained {unimproved for over 30 years was} the \defn{infinite barrier}
	that was also introduced in the original paper \cite[$\S$7]{BFT95}, where \eqref{eq:BFT}
	was extended to infinite posets.  The authors then considered an example of an
	\defng{infinite Fibonacci poset} (see below), first introduced in \cite{Bri88}, and
	for which \eqref{eq:BFT} is sharp.  Despite many recent developments in both geometric
	and combinatorial inequalities, the infinite barrier gave little hope for improvement, 
	see e.g.\ \cite{Bri99}. 
	
	\smallskip

	\subsection{Main result}\label{ss:intro-new}
	In this paper, we overcome this infinite poset barrier and obtain the first
	improvement on the bound in \eqref{eq:BFT}, as stated in the following theorem.

	\smallskip
	
	\begin{thm}[{\rm \defna{Breaking the infinite barrier}}{}]\label{thm:main}
		There exists a fixed constant $\ep>0$, such that for every finite poset
		\. $P:=(X,\prec)$ \. that is not a chain, there exist
		elements \ts $x,y\in X$, s.t.\
		\begin{equation}\label{eq:main-thm}
			\frac{5-\sqrt{5}}{10}  \. + \. \ep   \ \leq \ \Pb[ f(x) \.  < \. f(y)]  \ \leq \ \frac{5+\sqrt{5}}{10} \. - \. \ep \,,
		\end{equation}
		where the probability is over uniform random \ts $f\in \cE(P)$.
	\end{thm}
	
	\smallskip
	
	The constant $\ep$ can be made explicit, but is microscopically small.
	As our primary goal is to establish a separation between finite and infinite posets 
	in this context, we make no attempt to optimize this constant.  For simplicity, we 
	omit the explicit derivation of~$\ep$.

	The proof of Theorem~\ref{thm:main} relies on three  main ingredients, each of 
	independent interest.  Indeed, we divide the posets into three classes:
	
	\smallskip
	
	$\circ$ \. posets of small range,
	
	$\circ$ \. posets of large width, and
	
	$\circ$ \. posets of large range and small width.  
	
	\smallskip
	
	\nin
	We treat these cases completely differently:  we break the infinite
	barrier for the first case, we prove the \. $\frac{1}{3}$--$\frac{2}{3}$ \. Conjecture 
	for the second case, and we use the asymptotic bound by Kahn and the first author 
	\cite{AK25b}, to obtain the result in the third case.  We now proceed to discuss  
	each case separately, give formal statements, and present the background 
	(cf.~$\S$\ref{ss:finrem-Rota}).  
	
	\smallskip

	\subsection{Posets of small range}\label{ss:intro-proof}
	To state them, we first define
	the \defnb{balance constant} \ts of elements $x,y\in X$ of a finite poset $P=(X,\prec)$ as 
	\[ \delta(P;x,y) \ := \  \min \big \{ \. \Pb[f(x) < f(y)], \, \Pb[f(y) < f(x)] \. \big \}, \]
	and of the poset itself: 
	\[\delta({{P}}) \ := \ \max_{{x,y\in X,\ x\neq y}} \. \delta(P;x,y).
	\]
	Here and throughout the paper, the probability is over uniform random $f\in \cE({{P}})$. 
	
	Denote \. $\BFT:=\frac{5-\sqrt5}{10}$.
	In this notation, the $\frac{1}{3}$--$\frac{2}{3}$ Conjecture asserts 
	that $\delta({{P}})\geq 1/3$ for every finite poset ${{P}}$ that is not a chain.
	Similarly, Theorem~\ref{thm:main} says that there exists an absolute constant \. $\ep>0$ \. such that
	\. $\delta({{P}}) \geq \BFT + \ep$ \. 
	for every such~$P$.

	\smallskip	
	For $x \in X$, let $\pi(x)$ denote the number of elements of $P$ incomparable to $x$.
	Define the \defnb{range}\footnote{For
		\ts $D\geq 2$, a poset ${P}$ is called \defnb{$D$-thin} \ts if \. $\pi({P})\leq D$.
		We prefer \emph{range} over \emph{thinness} which is more standard in the literature.}
	of~$P$ as 
	$$\pi(P) \, := \, \max_{x \in X} \. \pi(x).
	$$ 
	The first ingredient in the proof of  Theorem~\ref{thm:main} is 
	an explicit bound on the balance constant for {posets of small range}.
	\smallskip
	
	\begin{thm}[{\rm \defna{The case of small range}}{}]\label{thm:ACPP}
		For an integer \ts $D \geq 2$, 
		let \. $P=(X,\prec)$ \. be a finite poset that is not a chain, such that \ts 
		$\pi(P)\leq D$. 
		Then:
		\[ \delta(P) \ > \ \frac{5-\sqrt{5}}{10} \. + \. \eta_D \quad \text{where}  \quad
		\eta_D  \ := \  \frac{(D+1)^{-3(D+1)}}{12}.  \]
	\end{thm}
	
	\smallskip
	
	We will prove Theorem~\ref{thm:ACPP} in Section~\ref{s:thm:ACPP}.
	Note that for \ts $D\leq5$, Brightwell and Wright \cite{BW92} proved a 
	stronger bound \. $\delta({P})\geq \frac13$, establishing the \. 
	$\frac13$--$\frac23$ \. Conjecture for such posets.  
	In~\cite{Pec08}, Peczarski extended this result to \ts $D\leq6$.
	Unfortunately, their methods break down for larger~$D$.  
	
	\smallskip
	
	We emphasize that the finiteness assumption in Theorem~\ref{thm:ACPP} is essential.
	Define the \defn{infinite Fibonacci poset} \ts $\cF=(X,\prec)$ \ts on \ts $X=\{x_i:i\in\mathbb Z\}$, 
	where \ts $x_i\prec x_j$ \ts if and only if \ts $j-i\geq2$.
	Let \ts $\cF_m$, where \ts $m\geq 1$, be the (finite) \defn{Fibonacci poset} \ts obtained by 
	restriction of \ts $\cF$ \ts 
	to a finite subset \ts $X_m:=\{x_{-m},\ldots,x_m\}$, cf.~\cite{MPP-phi}.  Define
	$$\delta(\cF;x_i,x_j) \, := \, \lim_{m\to \infty} \delta(\cF_m;x_i,x_j)
	$$
	and let 
	\[\delta({{\cF}}) \ := \ \sup_{{i,j\in \zz, \. i<j}} \. \delta(\cF;x_i,x_j).
	\]
	This   poset satisfies \. $\pi({\cF})=2$, since each element is 
	incomparable with exactly two others, and yet it can be shown that
	\. $\delta({\cF})=\BFT$, see~\cite[\S7]{BFT95}.\footnote{
		As one would expect, we have \. $\max \delta(\cF_m;x_i,x_j) \ge \frac13$.  This is a typical case  when the order of operations matters: \. $\lim_{m\to \infty} \max_{i<j} \delta(\cF_m;x_i,x_j) \ne \sup_{i<j} \lim_{m\to \infty}  \delta(\cF_m;x_i,x_j)$.}

	Our proof builds on the approach of Brightwell--Felsner--Trotter~\cite{BFT95},
	which in turn builds on the approach of Kahn--Saks~\cite{KS84}.
	The original strategy is to find a special triple of elements {called a \emph{BFT triple}},
	and the BFT analysis shows that it contains a pair with balance at least~$\BFT$.
	Our strategy is to improve this bound by choosing the BFT triple more carefully.
	
	The finite restrictions of the Fibonacci poset illustrate why this choice matters.
	Consider the restriction to $X_m$ and note that 
	as $m\to\infty$, the adjacent pairs at any fixed $(i,j)$, the balance constant \ts 
	$\delta(\cF_m;x_i,x_j)\to \BFT \approx 0.2764$, while for the pairs \. $(i,j)=(m-1,m)$ \. at the endpoints 
	we have  
	\ts $\delta(\cF_m;x_{m-1},x_{m}) \to \frac{3-\sqrt{5}}{2}\approx 0.3820$.
	Thus these finite posets contain substantially better-balanced pairs near their boundaries.
	In the two-sided infinite limit, however, the endpoints recede to infinity, and these favorable pairs are lost.
	This explains the obstruction posed by the infinite example: choosing triples near the center \emph{cannot} yield a uniform improvement over \eqref{eq:BFT}.
	In other words, we need a way to distinguish the better-balanced triples from those that reproduce the behavior of the infinite poset.

	Our main new tool for making this distinction is the {displacement} of an element from its 
	position in the height order.  Formally, let 
	$$h(x)\, := \, \Eb[f(x)] 
	$$
	denote the \defn{height} \ts of $x$ in a uniform random linear extension \ts $f\in \cE(P)$. 
	Now let \ts $X=\{v_1,\ldots,v_n\}$, where elements $v_i$ are written in nondecreasing 
	order of~$h$.  Finally, define
	\[
	d_i \, := \, h(v_i)-i \qquad \text{and} 
	\qquad
	M \, := \, \max_{1\le i \le n} \. |d_i|\ts,
	\]
	the \defnb{displacement} and \defnb{maximal absolute displacement}, respectively. 
	In $\cF_m$, the displacement $\to 0$ on every element $x_i$, as $m \to\infty$, 
	while the maximal absolute displacement remains bounded away from zero. 
	In other words, the displacement  detects the boundary behavior that would disappear 
	in the limit.
	
	Our strategy is therefore to choose a BFT triple with one of its elements attaining the 
	maximum absolute displacement $M$. 
	We then derive a lower bound on the best pair balance of the selected triple in terms of $M$.
	In turn, $M$ admits a positive lower bound depending only on the range $\pi({P})$, which explains the role of the range assumption in the theorem.
	Together, these bounds yield Theorem~\ref{thm:ACPP}.
	Establishing them accounts for most of the technical work in the proof.

	\smallskip
	
	\subsection{Poset of large width}
	
	Note that the improvement over $\BFT$ in Theorem~\ref{thm:ACPP} tends to zero 
	as the range $\pi({P})$ grows: \. $\eta_D\to 0$ \. as \.  $D\to \infty$.    
	Thus, we need a completely different strategy in this case.  
	Recall that the \defnb{width} $w({P})$  of a poset ${P}$ is the maximum 
	size of an antichain (set of pairwise incomparable elements of ${P}$).

	\smallskip
	
	\begin{thm}[{\rm \defna{The case of large width}}{}]\label{thm:Aires}
		There exists a universal constant $K>0$, such that for 
		every finite poset $P=(X,\prec)$ \. of width \ts $w(P)>K$,
		we have:
		\[ \delta(P)  \ >  \ e^{-1} \. - \. 10^{-100}.  \]
	\end{thm}
	
	\smallskip
	
	We  prove Theorem~\ref{thm:Aires} in Section~\ref{s:proof:Aires}.
	Note that \. $e^{-1} \approx {0.3679}  \. >  \. \frac13\ts$, so the result implies the \. $\frac13$--$\frac23$ \. Conjecture 
	in this case.  
	The constant $e^{-1}$ in Theorem~\ref{thm:Aires} is an artifact of our use
	of Gr\"unbaum-type inequalities~\cite{Gru60} as the log-concavity input in the proof. 
	The celebrated \defn{Kahn--Saks conjecture}  \cite[p.~114]{KS84}  claims 
	a stronger conclusion, that \. $\delta({P})\to \frac12$ \. as \. $w({P})\to\infty$.
	This conjecture was recently resolved by the first author in a companion
	paper~\cite{Air26} (in fact, that paper gave an inspiration for this work),
	and the proof  requires substantially more involved technical estimates.
	In particular, it replaces the log-concavity input used here with the more refined
	\defnb{Koml\'os selection theorem} of~\cite[Thm.~8.1]{AK25a}, itself inspired by the ingenious  work
	of Koml\'os~\cite{Kom90}.    
	We present the $e^{-1}$ bound here because it admits a much simpler proof
	and already suffices for our purposes.

	The proof of Theorem~\ref{thm:Aires} centers on the following statistic
	of $P$. As before, list the elements as $v_1,\ldots,v_n$ in nondecreasing
	order of height, and define
	\[
	\gap(P) \quad := \quad
	\max\bigl\{h(v_1),\ h(v_{i+1})-h(v_i)\ (1\leq i<n),\
	n+1-h(v_n)\bigr\},
	\]
	the largest  gap between consecutive terms of
	$0,h(v_1),\ldots,h(v_n),n+1$.
	Our proof introduces two distinct, entirely new strategies for choosing a balancing pair.

	For the first strategy, we  show   that there exist two consecutive
	elements $x,y\in X$ in the height order satisfying
	\[
	\delta(P;x,y)
	\quad \geq \quad
	e^{-1}-\frac{2\sqrt{3}\,\gap(P)}{w(P)}.
	\]
	This lower bound follows from a short argument combining the \emph{log-concave
		estimates} of Lov\'asz--Vempala~\cite{LV07} with the \emph{window estimates} of Aires--Kahn~\cite{AK25a}.
	This gives the desired  bound when the ratio $\gap(P)/w(P)$
	is \defna{relatively small}.
	
	For the second strategy, under some technical assumptions, we construct an antichain containing distinct elements $x,y$
	satisfying
	\[
	\delta(P;x,y)
	\quad \geq \quad
	e^{-1} \ - \  C\.
	\frac{(\log\gap(P))^2}{\gap(P)},
	\]
	for some absolute constant $C>0$.
	When the ratio $\gap(P)/w(P)$ is \defna{relatively large},
	the lower bound $w(P)>K$ forces $\gap(P)$ to be large as well.
	Since $(\log t)^2/t$ tends to zero as $t$ grows, this estimate
	then gives the desired lower bound.
	Constructing such an antichain accounts for the bulk of the technical
	work. In addition to the estimates of Lov\'asz--Vempala~\cite{LV07}
	and Aires--Kahn~\cite{AK25a} used above, we rely on Shepp's
	\emph{XYZ correlation inequality}~\cite{She82}
	(see also~\cite{AS16,CP-LE}) and Haqi's remarkable recent proof of the
	\emph{Haqi--Kahn ideal inequality}~\cite{Haq26}.

	\smallskip
	
	\subsection{Poset of small width and large range}
	It remains to show that Theorem~\ref{thm:main} holds for posets with \emph{small width} 
	and \emph{large range}.  This is achieved by the following theorem of Aires and 
	Kahn~\cite{AK25b} which gives a much stronger asymptotic bound than we need:
	
	\smallskip
	
	\begin{thm}[{\textnormal{\defna{The case of small width and large range}}}~{\cite[Thm~1.6]{AK25b}}]\label{thm:AK}
		For all \ts $K,\eps>0$, there exists \ts $L=L(K,\eps)>0$, such that 
		for every finite poset $P=(X,\prec)$ \. of width \ts $w(P)< K$ \ts 
		and range \ts $\pi(P)>L$, we have:	
		$$
		\delta(P) \. > \. \frac12 \. - \. \eps\ts.
		$$
	\end{thm}		
	
	\smallskip
	
	In other words, for posets of bounded width \ts $w(P)<K$, we have:
	\[  \pi(P) \. \to \. \infty \quad \ \ \Longrightarrow \quad \ \ \delta(P) \. \to \. \frac12\..    \]
	Clearly, by taking \ts $\eps=\frac{1}{6}$ \ts we obtain the \. $\frac{1}{3}$--$\frac{2}{3}$ \. 
	Conjecture for all posets of bounded width and large enough range.  
	Theorem~\ref{thm:main} now follows from combining Theorem~\ref{thm:ACPP}, 
	Theorem~\ref{thm:Aires}, and Theorem~\ref{thm:AK}.

	We note that all constants in Theorems~\ref{thm:Aires} and~\ref{thm:AK} can be made explicit.
	We omit their values  from the theorem statements for readability, as determining them is not a primary aim of this paper, 
	see also~$\S$\ref{ss:finrem-hope}.

	\smallskip
	
	\subsection{Paper structure}
	The paper is organized as follows.
	In Section~\ref{s:prelim-ACPP}, we review results in the literature, and 
	in Section~\ref{s:small-range} we establish new technical lemmas we need 
	towards the proof of Theorem~\ref{thm:ACPP}.
	The proof of Theorem~\ref{thm:ACPP} is given in Section~\ref{s:thm:ACPP}. 
	Then, in Section~\ref{s:prelim-Aires}, we review the results in the literature 
	towards the proof of Theorem~\ref{thm:Aires}.  The proof itself is given in 
	Section~\ref{s:proof:Aires}.  We conclude with final remarks in Section~\ref{s:finrem}.

	\medskip

	\section{Preliminaries for Theorem~\ref{thm:ACPP} }\label{s:prelim-ACPP}
	
	
	
	\subsection{Preliminaries}

	Throughout this paper, let \. $P=(X,\prec)$ \. be a finite poset,
	where $X$ is the ground set and $\prec$ is a partial order on $X$.
	Write $n:=|X|$ and $[n]:=\{1,\ldots,n\}$.
	We use $f$ to denote a uniformly random linear extension of $P$,
	and $\Pb$ and $\Eb$ to denote the corresponding probability measure
	and expectation, respectively.

	Let
	\[ \BFT \ := \  \frac{5-\sqrt{5}}{10} \ \approx \ {0.276393\ldots}    \]
	We begin with a lemma that will be used repeatedly throughout the next three sections.		
	\smallskip
	
	\begin{lemma}[{\cite[p.~365]{KL90}},{\cite[Cor.~5.1(b)]{AK25a}}]\label{lem:KL}
		Let $x,y$ be distinct elements satisfying $h(x)\leq h(y)$.
		Then
		\[  \Pb[\cL(y) < \cL(x) ]  \ \leq \  1-e^{-1}. \]
	\end{lemma}
	
	\smallskip
	This result was first proved by Kahn and Linial~\cite[p.~365]{KL90}  via Gr\"unbaum-type inequalities~\cite{Gru60}.
	Note that they stated the result  for the case $h(x)=h(y)$, although their proof also covers the general case.
	A proof of the general statement appears in~\cite[Cor.~5.1(b)]{AK25a}.
	The following consequence of Lemma~\ref{lem:KL} will be especially useful.

	\smallskip
	
	\begin{cor}\label{lem:comp}
		Suppose that $\delta(P) <  e^{-1}$. Then,
		for all distinct $x,y \in X$ satisfying $h(x) \leq h(y)$,
		\[ \delta(P;x,y) \ = \  \Pb[\cL(y)< \cL(x)].  \]
	\end{cor}
	
	\smallskip
	
	\begin{proof}
		Suppose to the contrary that the claim is false.
		It then follows from Lemma~\ref{lem:KL} that  
		\[ \delta(P;x,y) \quad = \quad   \Pb[\cL(x)< \cL(y)] \quad   \geq \quad  e^{-1},\]
		contradicting the assumption that $\delta(P)< e^{-1}$.
	\end{proof}

	\smallskip

	We now   review the results of Brightwell, Felsner, and Trotter~\cite{BFT95} that will be used in the proof of Theorem~\ref{thm:ACPP}.
	For an element $x$ of a poset ${P}$, the \defnb{height} $h(x):=\Eb[\cL(x)]$ is the expected position of $x$ in a uniformly random linear extension $\cL$ of ${P}$.
	We call a triple $(x,y,z)$ of distinct elements of ${P}$ a \defnb{BFT triple} if these elements do not form a chain and
	\[
	h(x) \ \leq \ h(y) \ \leq \ h(z) \ \leq \ h(x)+2.
	\]
	The use of height functions in this setting goes back to the breakthrough work of Kahn and Saks~\cite{KS84}, who used pairs of incomparable elements $x,y$ satisfying $h(x)\leq h(y)\leq h(x)+1$ to establish the bound \[ \delta({P};x,y)\geq 3/11\approx 0.2727. \] Building on this approach, Brightwell, Felsner, and Trotter~\cite{BFT95} showed that a BFT triple $(x,y,z)$ yields
	\[ \max\big\{\delta({P};x,y), \. \delta({P};x,z), \. \delta({P};y,z)\big\} \ >  \  \BFT \  \approx \ {0.276393\ldots}. \]
	In this paper, we refine this approach by showing that every finite nonchain poset ${P}$ with at least three elements and \. $\pi({P})\leq D$ \. contains a BFT triple $(x,y,z)$ satisfying the stronger bound \[ \max\big\{\delta({P};x,y), \. \delta({P};x,z), \. \delta({P};y,z)\big\} \ > \ \BFT+ \frac{(D+1)^{-3(D+1)}}{12}, \] thereby proving Theorem~\ref{thm:ACPP}.


	Write \. $u \parallel v$ \. if $u$ and $v$ are incomparable in ${P}$.
	Following~\cite[Section~4]{BFT95}, and passing to the dual poset and reversing the order of the triple if necessary, we {split the analysis into} the following four cases:
	\begin{enumerate}[label=\textnormal{Case \Alph*.},leftmargin=*]
		\item $x \prec z$, \. $y \prec z$, \. and \. $x \parallel y$.
		\item $y \prec z$, \. $x \parallel y$, \. and \. $x \parallel z$.
		\item $x,y,z$ are pairwise incomparable.
		\item $x \prec z$, \. $x \parallel y$, \. and \. $y \parallel z$.
	\end{enumerate}

	Cases~A and~B were treated in~\cite[Thms.~4.1 and~4.2]{BFT95}, where it was shown that
	\[
	\Pb[ \cL(y) \. < \. \cL(x)] \quad \geq \quad \frac{1}{3}.
	\]
	Together with Lemma~\ref{lem:KL}, this gives\footnote{Note that the upper bound $2/3$ for $\Pb[\cL(y)<\cL(x)]$ and
		$\Pb[\cL(z)<\cL(y)]$ was stated in~\cite[p.~348]{BFT95}.
		This bound can be deduced from their methods, although its
		justification was not given explicitly. Here we use the
		stronger upper bound $1-e^{-1}$ provided by
		Lemma~\ref{lem:KL}.}
	\[
	\frac{1}{3} \quad \leq \quad \Pb[ \cL(y) \. < \. \cL(x)]
	\quad \leq \quad 1-e^{-1} \quad < \quad \frac{2}{3}.
	\]
	Thus, in either case, \. $\delta({P};x,y)\geq 1/3$\..
	In Case~C, Brightwell, Felsner, and Trotter~\cite[Sec.~6]{BFT95} established the bound $\delta(P)\geq \BFT$, which is insufficient for the conclusion of Theorem~\ref{thm:ACPP}.
	Nevertheless, their argument yields the following slightly stronger bound.

	\smallskip
	
	\begin{lemma}[{\rm \defna{Case C of BFT}~{\cite[\S6]{BFT95}}}{}]\label{lem:BFT-C}
		Let $(x,y,z)$ be a BFT triple whose elements are pairwise incomparable.
		Then
		\[
		\max \big\{ \delta(P;x,y), \. \delta(P;y,z) \big\} \  >   \ \frac{67}{242}  \ \approx \ {0.27686}.
		\]
	\end{lemma}
	
	\smallskip
	
	We include the details to verify that their argument yields the bound $\frac{67}{242}$ in Appendix~\ref{s:BFT-C}.

	


	\subsection{Case D of  Brightwell--Felsner--Trotter analysis}\label{ss:case-D}
	
	Our treatment of Case~D requires new arguments involving the range $\pi(P)$.
	We develop these in Sections~\ref{s:small-range} and~\ref{s:thm:ACPP},
	building on the following quantitative analysis of Case~D.

	Let $(x,y,z)$ be a BFT triple in Case~D, i.e. $h(x)\leq h(y) \leq h(z) \leq h(x)+2$ and
	\[
	x\prec z, \qquad x\parallel y, \qquad y\parallel z.
	\]
	For $i,j\in\mathbb{Z}$, let
	\[
	p(i,j) \quad := \quad
	\Pb[\cL(y)-\cL(x)=i,\ \cL(z)-\cL(y)=j].
	\]
	We also write
	\[
	p \ := \ \Pb[\cL(y) \. < \. \cL(x)],
	\qquad
	p' \ := \ \Pb[\cL(z) \. < \. \cL(y)],
	\qquad
	S \ := \ p+p',
	\]
	and set
	\[
	a_1 \ := \ p(1,1), \qquad
	a_2 \ := \ p(1,2)+p(2,1), \qquad
	a_3 \ := \ p(2,2).
	\]
	The inequalities in~\cite[(5.2), (5.5), (5.7)--(5.8)]{BFT95} give
	\begin{align}
		2a_1+a_2 &\ \leq \ S,
		& 2a_2+2a_3 &\ \leq \ S, \label{eq:BFT-D-linear}\\
		4a_1a_3 &\ \leq \ a_2^2,
		& 7 &\ \leq \ 9S+7a_1+3a_2+3a_3. \label{eq:BFT-D-quadratic}
	\end{align}
	In particular, the quadratic inequality in {\eqref{eq:BFT-D-quadratic}} follows from the special case of the cross-product inequality proved in~\cite[Thm.~3.2]{BFT95},
	and the general case of the inequality~\cite[Conj.~3.1]{BFT95} remains a challenging open problem.
	Let
	\begin{equation}\label{eq:BFT-D-slack}
		R \quad := \quad S-2a_1-a_2.
	\end{equation}
	Note that we have 
	\begin{equation}\label{eq:RS}
		0 \ \leq \  R \ \leq \ S  \ = \ p+p'  \ \leq \ 1,    
	\end{equation}
	where the first  inequality is because of \eqref{eq:BFT-D-linear}, and the last inequality is because 
	the events \. $\{\cL(y) < \cL(x) \}$ \. and \. $\{\cL(z) < \cL(y)\}$ \. are disjoint (because of the assumption $x \prec z$).
	The next lemma shows that a lower bound for $R$ will give  an improvement to  the BFT bound.
	
	\smallskip
	
	\begin{lemma}[{\rm \defna{Case D of BFT}}{}]\label{lem:BFT-D}
		Let $(x,y,z)$ be a BFT triple satisfying Case D. Then
		\begin{equation}\label{eq:BFT-D-stability}
			\max\big\{\delta(P;x,y), \. \delta(P;y,z)\big\}
			\quad \geq \quad \BFT+\frac{R}{5+3\sqrt{5}},
		\end{equation}
		where $R$ is as defined in \eqref{eq:BFT-D-slack}.
		Moreover,
		\begin{equation}\label{eq:BFT-D-nonadjacent}
			R \quad \geq \quad
			\Pb[\cL(x)-\cL(y)\geq2]
			\ + \ \Pb[\cL(y)-\cL(z)\geq2].
		\end{equation}
	\end{lemma}
	
	\smallskip
	
	In~\cite[\S5]{BFT95}, only the weaker bound obtained by dropping the term involving $R$ was used, whereas we will need to  use the full strength of Lemma~\ref{lem:BFT-D} in the proof of Theorem~\ref{thm:ACPP}.
	The proof of this  lemma closely follows the argument in~\cite[\S5]{BFT95}.
	
	\smallskip
	
	\begin{proof}
		Since $x\parallel y$ and $y\parallel z$, we have $S>0$.
		Set
		\[
		b_i \ := \ \frac{a_i}{S} \quad (i=1,2,3),
		\qquad
		s \ := \ 1-\frac{R}{S}.
		\]
		Then $s\in[0,1]$, and~\eqref{eq:BFT-D-linear}--\eqref{eq:RS} imply
		\begin{equation}\label{eq:BFT-D-normalized}
			2b_1+b_2 \ = \ s, \qquad
			b_2+b_3 \ \leq \ \frac{1}{2}, \qquad
			4b_1b_3 \ \leq \ b_2^2.
		\end{equation}

		Let $F(s)$ (for $s \in [0,1]$) be given by
		\[
		F(s) \quad := \quad
		\max\left\{\frac{7}{2}s,\,
		\frac{7}{2}\sqrt{s^2+\frac{1}{4}}-\frac{1}{4}\right\}.
		\]
		
		\begin{claim}
			We have
			\begin{equation}\label{eq:BFT-D-objective}
				7b_1+3b_2+3b_3 \quad \leq \quad F(s).
			\end{equation}
		\end{claim}
		
		\begin{proof}
			When $b_1=0$, the constraint \. $b_2+b_3\leq1/2$ \. gives
			\[
			7b_1+3b_2+3b_3
			\quad = \quad 3(b_2+b_3)
			\quad \leq \quad \frac{3}{2}.
			\]
			On the other hand, by the definition of $F(s)$,
			\[
			F(s)
			\quad \geq \quad
			\frac{7}{2}\sqrt{s^2+\frac{1}{4}}-\frac{1}{4}
			\quad \geq \quad
			\frac{7}{2}\cdot\frac{1}{2}-\frac{1}{4}
			\quad = \quad \frac{3}{2}.
			\]
			Combining these bounds proves~\eqref{eq:BFT-D-objective} when $b_1=0$.
			
			Otherwise, write $u:=b_1>0$.
			By~\eqref{eq:BFT-D-normalized}, we have $u\leq s/2$ and
			\[
			b_2+b_3
			\quad \leq \quad b_2+\frac{b_2^2}{4u}
			\quad = \quad \frac{s^2}{4u}-u.
			\]
			It follows that
			\[
			7b_1+3b_2+3b_3
			\quad \leq \quad
			7u+3\min\left\{\frac{1}{2},\,\frac{s^2}{4u}-u\right\}.
			\]
			The two expressions inside the minimum are equal when $u$ is equal to
			\[
			u_0 \quad := \quad \frac{\sqrt{1+4s^2}-1}{4}.
			\]
			For $0<u\leq u_0$, the right side is $7u+3/2$, which is increasing in $u$.
			For $u_0\leq u\leq s/2$, it is $4u+3s^2/(4u)$, which is convex in $u$ and hence is maximized at an endpoint.
			Evaluating at \. $u=u_0$ \. and \. $u=s/2$ \. proves~\eqref{eq:BFT-D-objective}.
		\end{proof}

		Now, note that 
		the function $F$ is convex on $[0,1]$, with
		\[
		F(0) \ = \ \frac{3}{2},
		\qquad
		F(1) \ = \ \frac{7\sqrt{5}-1}{4} \ =: \ c.
		\]
		Thus $F$ lies below the chord joining these two endpoint values, so
		\[
		F(s) \quad \leq \quad
		c-\left(c-\frac{3}{2}\right)(1-s).
		\]
		Combining this with~\eqref{eq:BFT-D-quadratic} and~\eqref{eq:BFT-D-objective}, we obtain
		\begin{align*}
			7
			&\ \leq \ 9S+S(7b_1+3b_2+3b_3) \ \leq \ S\bigl(9+F(s)\bigr) \\
			&\ \leq \ (9+c)S-\left(c-\frac{3}{2}\right)R.
		\end{align*}
		Since
		\[
		\frac{7}{2(9+c)} \ = \ \BFT,
		\qquad
		\frac{c-3/2}{2(9+c)} \ = \ \frac{1}{5+3\sqrt{5}},
		\]
		it follows that
		\[
		\frac{S}{2} \quad \geq \quad \BFT+\frac{R}{5+3\sqrt{5}}.
		\]        
		This implies that        
		\begin{equation}\label{eq:HAL-1}
			\max\{p,p'\} 
			\quad \geq \quad \frac{S}{2} \quad \geq \ \quad \BFT+\frac{R}{5+3\sqrt{5}}.
		\end{equation}
		On the other hand, we have 
		\[ \max\{p,p'\}  \quad  \leq \quad   1-{e^{-1}} \quad < \quad 1-\big(\BFT+\frac{1}{5+3\sqrt{5}}\big) \quad  \leq \quad   1-\big(\BFT+\frac{R}{5+3\sqrt{5}}\big),  \]
		where the first inequality is because of Lemma~\ref{lem:KL}, and the last inequality is because of \eqref{eq:RS}.
		Combining this inequality with \eqref{eq:HAL-1}  proves \eqref{eq:BFT-D-stability}.

		To prove~\eqref{eq:BFT-D-nonadjacent}, note that swapping adjacent incomparable elements gives
		\[
		p(-1,j)=p(1,j-1), \qquad
		p(j,-1)=p(j-1,1) \qquad (j\geq2).
		\]
		Consequently,
		\[
		2a_1+a_2
		\quad = \quad p(-1,2)+p(-1,3)+p(2,-1)+p(3,-1).
		\]
		The four events on the right are disjoint and are contained in the two  events defining $S$ (i.e. \. $\{\cL(x)-\cL(y)\geq 1\}$ \. and \. $\{\cL(y)-\cL(z)\geq 1\}$\.).
		These two  events are themselves disjoint, since $x\prec z$.
		Subtracting the four events from these two events {immediately implies~}\eqref{eq:BFT-D-nonadjacent}.
	\end{proof}
	
	\medskip

	\section{Posets of small range}\label{s:small-range}
	
	Let $D\geq 2$ throughout this section.
	In this section, we start using the assumption \. $\pi({P})\leq D$ \. from Theorem~\ref{thm:ACPP}.
	Recall that the \defnb{range} $\pi(P)$ is the
	maximum, over $x\in X$, of the number of elements incomparable with $x$.
	
	We begin with the following simple lemma.
	Let $S\subseteq X$ be a subset.
	We say that an event $E$ is \defnb{$S$-dependent} if it is specified by a fixed collection of comparisons among elements of $S$, that is,
	\[
	E  \ = \  \bigcap_{i=1}^{k} \{ \cL(u_i) < \cL(v_i) \}, \qquad \text{ where \   $u_i,v_i \in S$ \ for all \. $i \in [k]$\..}
	\]
	Write  
	\[q_0 \, := \, 1, \qquad q_m \, := \,  (D+1)^{-m(D+1)} \quad \text{for} \ \ m\geq 1,\]
	and note that 
	\begin{equation}\label{eq:eta}
		\eta_D  \ = \  \frac{(D+1)^{-3(D+1)}}{12} \ = \  \frac{q_3}{12}\,,    
	\end{equation}
	where $\eta_D$ is the constant in Theorem~\ref{thm:ACPP}.
	\smallskip

	\begin{lemma}\label{lem:feasible}
		{Let ${P}$ be a finite poset} satisfying $\pi({P})\leq D$, and let $S \subseteq X$.
		Let $E$ be an $S$-dependent event with $\Pb[E]>0$. Then
		\[
		\Pb[E] \quad \geq \quad q_m,
		\]
		where $m:=|S|$.
	\end{lemma}
	
	\smallskip
	
	\begin{proof}
		Set
		\[
		T \ := \ S\cup\{v\in X:\ v\parallel s\text{ for some }s\in S\},
		\qquad
		U \ := \ X\smallsetminus T.
		\]
		The assumption on the range gives
		\[
		|T| \quad \leq \quad \sum_{s\in S}\bigl(1+\pi(s)\bigr)
		\quad \leq \quad m(D+1).
		\]
		
		Let $\cQ$ be obtained from ${P}$ by adjoining the comparisons specified by $E$ and taking the transitive closure.
		Since $\Pb[E]>0$, these comparisons are consistent with the partial order of ${P}$, so $\cQ$ is a well-defined poset.		
		{We denote the strict partial orders of ${P}$ and $\cQ$ by $\prec_{{P}}$ and $\prec_{\cQ}$, respectively.}
		
		We first show that, for all \. $u,v\in U$\.,
		\[
		u \. \prec_{{\cQ}} \.  v \quad \Longrightarrow \quad u \. \prec_P \. v.
		\]
		Indeed,
		if \. $u \. \prec_P \. v$\., there is nothing to prove.
		Otherwise, a path witnessing $u\prec_{{\cQ}} v$ must use an added
		comparison from $E$, so there exists $s\in S$ such that
		\[
		u \. \prec_{{\cQ}} \. s \. \prec_{{\cQ}}  \. v.
		\]
		By the definition of $U$, both $u$ and $v$ are comparable with $s$ in ${P}$.
		Since $\cQ$ extends ${P}$, these comparisons must have the same orientations in both posets.
		Hence
		\[
		u \. \prec_P \.  s \. \prec_P  \. v,
		\]
		which implies \.  $u \. \prec_P  \. v$, as desired.
		In particular, this implies that
		\begin{equation}\label{eq:qupu}
			\cQ[U] \quad = \quad {P}[U],
		\end{equation}
		where ${P}[U]$ denotes the subposet induced by $U$, with the order inherited from $\prec_{P}$.
		
		For a finite poset $R$, let $e(R)$ denote its number of linear extensions.
		It follows from \eqref{eq:qupu}  that
		\begin{equation}\label{eq:feasible-lower}
			e(\cQ) \quad \geq \quad e(\cQ[U])  \quad = \quad e({P}[U]).
		\end{equation}
		
		We next give an upper bound for $e(P)$.
		For $t\in T$, let
		\[
		d(t) \ := \ \bigl|\{v\in X:v\prec_P t\}\bigr|.
		\]
		In any linear extension $g$ of $P$, we have
		\[
		d(t)+1 \quad \leq \quad g(t)
		\quad \leq \quad d(t)+1+\pi(t).
		\]
		Thus $g(t)$ has at most \. $\pi(t)+1\leq D+1$ \. possible values.
		Once the relative order on $U$ is fixed, the positions $g(t)$
		for $t\in T$ determine the entire linear extension, as the elements
		of $U$ fill the remaining positions in their prescribed order.
		Consequently,
		\begin{equation}\label{eq:feasible-upper}
			e({P}) \quad \leq \quad e({P}[U])\,(D+1)^{|T|}.
		\end{equation}
		
		Combining~\eqref{eq:feasible-lower} and~\eqref{eq:feasible-upper}, we therefore obtain
		\[
		\Pb[E]
		\quad = \quad \frac{e(\cQ)}{e({P})}
		\quad \geq \quad (D+1)^{-|T|}
		\quad \geq \quad (D+1)^{-m(D+1)}
		\quad = \quad q_m,
		\]
		as desired.
	\end{proof}
	
	\smallskip
	
	We now combine Lemma~\ref{lem:feasible} with Lemma~\ref{lem:BFT-D}
	to describe the structure of BFT triples when $\delta({P})$ is close to $\BFT$.
	Let $G({P})$ denote the \defnb{incomparability graph} of ${P}$,
	whose vertex set is $X$ and whose edges join incomparable elements.
	An edge is a \defnb{bridge} if deleting it increases the number of connected components.
	For subsets $V,W\subseteq X$, we write \. $V\prec W$ \. if every element of $V$ precedes every element of $W$ in ${P}$.
	For two distinct elements $x, y \in X$,
	we say  that $y$ \defnb{covers} $x$, denoted by
	\. $x\lessdot y$\.,  if \. $x\prec y$ \. and no element $w\in X$ satisfies \. $x\prec w\prec y$\..
	{We use the same notation when either set consists of a single element, identifying that singleton with its element.}
	
	\smallskip
	
	\begin{lemma}[{\rm \defna{Structure of BFT triples}}{}]\label{lem:rigidity}
		Let ${P}$ be a finite poset satisfying
		\[
		\pi({P}) \ \leq \  D
		\qquad \text{ and } \qquad
		\delta({P}) \ \leq \ \BFT+\eta_D,
		\]
		and let $(x,y,z)$ be a BFT triple.
		Then the following statements hold:
		\begin{enumerate}[label=\textnormal{(\roman*)}]
			\item The triple $(x,y,z)$ is in Case~D;
			\item 	$z$ covers $x$;
			\item  $y$ is incomparable only with $x$ and $z$;
			\item The edges $xy$ and $yz$ are bridges of $G({P})$;
			\item 	Writing
			\[
			L \ := \ \{w\in X:w\prec y\},
			\qquad
			U \ := \ \{w\in X:y\prec w\},
			\]
			we have the following partition of the ground set
			\begin{equation}\label{eq:bridge-partition}
				X \quad = \quad L\,\sqcup \,\{x,y,z\}\, \sqcup \,U,
			\end{equation}
			and
			\begin{equation}\label{eq:bridge-order}
				L\cup\{x\} \quad \prec \quad U\cup\{z\}.
			\end{equation}
		\end{enumerate}
	\end{lemma}
	
	\smallskip
	
	Intuitively, Lemma~\ref{lem:rigidity} states that, under its hypotheses, every BFT triple induces a copy of the  $(2+1)$ poset  and interacts with the rest of the poset in a highly constrained way.
	Here, \defnb{$(2+1)$-poset} is the disjoint union of a two-element chain
	and a singleton. Equivalently, it consists of three elements $x,y,z$
	with $x\prec z$ and $y$ incomparable to both $x$ and $z$.
	{Since the $(2+1)$ poset itself} has balance constant $1/3$, such triples are natural places to look for well-balanced pairs.

	\begin{proof}
		We start {by proving} (i).
		Since $D\geq2$, we have
		\[
		\delta({P})
		\quad \leq \quad \BFT+\eta_D
		\quad \leq \quad \BFT+\frac{3^{-9}}{12}
		\quad < \quad 0.2764.
		\]
		Cases~A and~B, including their duals, give a pair with balance constant at least $1/3$, contradicting $\delta({P})<0.2764$.
		Case~C is excluded by Lemma~\ref{lem:BFT-C}.
		Thus the triple is in Case~D, proving~\textnormal{(i)}.
		
		To prove~\textnormal{(ii)}, suppose to the contrary that there exists $w$ satisfying $x\prec w\prec z$.
		Then every linear extension $g$ of $P$ satisfies
		\[
		g(z)-g(x) \quad \geq \quad 2.
		\]
		Since $x\parallel y$ and $y\parallel z$, the event
		\. $\{\cL(x)<\cL(y)<\cL(z)\}$ \. has positive probability (recall that $\cL$ is the uniform random linear extension of $P$).
		On this event, both $w$ and $y$ lie between $x$ and $z$, so the gap is at least three.
		Taking expectations gives
		\[
		h(z)-h(x)
		\quad = \quad \Eb[\cL(z)-\cL(x)]
		\quad > \quad 2,
		\]
		contradicting the definition of a BFT triple.
		Thus \. $x\lessdot z$ \., proving~\textnormal{(ii)}.
		
		For~\textnormal{(iii)}, we first establish two consequences of the bound on $\delta({P})$.
		By Lemma~\ref{lem:BFT-D},
		\[
		\BFT+\frac{R}{5+3\sqrt{5}}
		\quad \leq \quad \delta({P})
		\quad \leq \quad \BFT+\eta_D.
		\]
		Consequently,
		\begin{equation}\label{eq:rigidity-small}
			R
			\quad \leq \quad (5+3\sqrt{5})\eta_D
			\quad = \quad
			\frac{5+3\sqrt{5}}{12}\,q_3     
			\quad < \quad q_3,
		\end{equation}
		where the equality is due to \eqref{eq:eta}, and the last inequality is due to  \. $5+3\sqrt{5} \. \approx \. 11.7$\..

		Suppose that some linear extension $g$ satisfies
		\. $g(y)<g(w)<g(x)$ \. for an element $w\in X\smallsetminus\{x,y\}$.
		Then the event
		\[
		E \ := \ \{\cL(y)<\cL(w)<\cL(x)\}
		\]
		is $\{x,y,w\}$-dependent and has positive probability.
		Lemma~\ref{lem:feasible} and~\eqref{eq:BFT-D-nonadjacent} give
		\[
		q_3
		\quad \leq \quad \Pb[E]
		\quad \leq \quad \Pb[\cL(x)-\cL(y)\geq2]
		\quad \leq \quad R,
		\]
		contradicting~\eqref{eq:rigidity-small}.
		Thus every linear extension $g$ satisfies
		\. $g(y)<g(x)$ \. only if \. $g(x)=g(y)+1$.
		The same argument applies to the pair $y,z$, so
		\begin{equation}\label{eq:bridge-adjacency}
			\begin{aligned}
				g(y)<g(x) &\quad \Longrightarrow \quad g(x)=g(y)+1,\\
				g(z)<g(y) &\quad \Longrightarrow \quad g(y)=g(z)+1.
			\end{aligned}
		\end{equation}
		
		The first condition in~\eqref{eq:bridge-adjacency} implies that,
		for every $w\in X\smallsetminus\{x,y\}$,
		\begin{equation}\label{eq:bridge-left}
			w\prec y \qquad \text{or} \qquad x\prec w.
		\end{equation}
		Indeed, if neither relation holds, then there  would  be a linear extension $g$  with
		$g(y)<g(w)<g(x)$, contradicting~\eqref{eq:bridge-adjacency}.
		Similarly, the second condition implies that,
		for every $w\in X\smallsetminus\{y,z\}$,
		\begin{equation}\label{eq:bridge-right}
			w\prec z \qquad \text{or} \qquad y\prec w.
		\end{equation}
		If an element $w\notin\{x,y,z\}$ were incomparable with $y$,
		then~\eqref{eq:bridge-left} and~\eqref{eq:bridge-right} would give
		\. $x\prec w\prec z$\..
		This contradicts~\textnormal{(ii)}.
		Thus $y$ is incomparable only with $x$ and $z$, proving~\textnormal{(iii)}.
		
		We now prove~\textnormal{(v)}.
		The partition in~\eqref{eq:bridge-partition} follows from~\textnormal{(iii)}.
		For $w\in L$, equation~\eqref{eq:bridge-right} gives $w\prec z$.
		For $w\in U$, equation~\eqref{eq:bridge-left} gives $x\prec w$.
		Also, $L\prec U$ follows by transitivity through $y$.
		Finally, these relations, together with $x\prec z$, give
		\[
		L\cup\{x\} \quad \prec \quad U\cup\{z\}.\]
		This proves~\eqref{eq:bridge-order}, and hence~\textnormal{(v)}.
		
		We now prove (iv).
		By (v),  there are no incomparability edges between $L\cup\{x\}$ and $U\cup\{z\}$.
		By~\textnormal{(iii)}, the only edges incident with $y$ are $xy$ and $yz$.
		Deleting $xy$ therefore separates $L\cup\{x\}$ from its complement,
		while deleting $yz$ separates $U\cup\{z\}$ from its complement.
		Thus both edges are bridges, proving~\textnormal{(iv)}.
	\end{proof}

	\medskip

	\section{Proof of Theorem~\ref{thm:ACPP}}\label{s:thm:ACPP}
	
	In this section, we prove Theorem~\ref{thm:ACPP}.
	Throughout this section, let ${P}$ be a finite nonchain poset satisfying \. $\pi({P})\leq D$\..
	We may assume without loss of generality that the incomparability graph $G({P})$ is connected.
	Indeed, the conclusion for ${P}$ follows from that for the subposet induced by any connected component with at least two vertices, since comparison probabilities within a component are unchanged.
	We will also assume that ${P}$ {has $n \geq 3$ elements}, as otherwise $\delta({P})=1/2$.

	\subsection{The strategy to choose the right BFT triple}\label{ss:choose-BFT}
	Our starting point is the approach introduced by  Kahn\textcolor{red}{--}Saks~\cite{KS84}, and refined by Brightwell--Felsner--Trotter~\cite{BFT95}.
	We list the elements of ${P}$ as $v_1,\ldots,v_n$ in nondecreasing order of height,
	\[
	h(v_1) \quad \leq \quad \cdots \quad \leq \quad h(v_n),
	\]
	breaking ties arbitrarily.
	The strategy of Brightwell--Felsner--Trotter is to find three consecutive elements $(v_k,v_{k+1},v_{k+2})$ that form a BFT triple.
	Applying the arguments outlined in Section~\ref{s:prelim-ACPP} to this triple then yields \. $\delta({P})\geq \BFT=\frac{5-\sqrt{5}}{10}$\..
	Our strategy builds on the BFT approach by choosing the BFT triple \emph{more carefully}.

	To see why the choice of triple matters, consider finite restrictions of the infinite \defnb{Fibonacci poset} { $\cF$ defined in Section~\ref{ss:intro-proof}}:
	Let $m$ be a positive even integer, and let $\cF_m$ be the poset with ground set $\{x_{-m},x_{-m+1},\ldots,x_{m-1},x_m\}$, with \. $x_i\prec x_j$ \. if and only if \. $j-i\geq2$\..
	The consecutive triples centered at odd indices are all BFT triples, but some can contain much better-balanced pairs than others.
	Indeed, writing $F_k$ for the Fibonacci numbers, with $F_0=0$ and $F_1=1$, a direct count gives
	\[
	\delta(\cF_m;x_i,x_{i+1})
	\quad = \quad
	\frac{F_{m+i+1}F_{m-i}}{F_{2m+2}}.
	\]
	It follows from direct calculation that, at the center ($i=0$),
	\[
	\delta(\cF_m;x_0,x_1)
	\quad \longrightarrow \quad
	\frac{5-\sqrt{5}}{10}
	\quad \approx \quad 0.276393,
	\]
	whereas
	at the boundary ($i=m-1$),
	\[
	\delta(\cF_m;x_{m-1},x_m)
	\quad \longrightarrow \quad
	\frac{3-\sqrt{5}}{2}
	\quad \approx \quad 0.381966.
	\]
	Thus, in this case, the best-balanced pairs lie near the boundary, while choosing pairs near the center is suboptimal.
	{As $m\to\infty$, the boundary recedes to infinity, while the balance of every adjacent pair at a fixed index converges to $\BFT$. The best-balanced pairs near the boundary are therefore lost in the limit.}
	This illustrates why the $1/3$--$2/3$ conjecture can fail for infinite posets and suggests how to improve the bound for finite posets:
	we must find a way to distinguish these better balanced pairs from the suboptimal ones.
	
	To identify promising BFT triples, we introduce the \defnb{displacement} of an element from its position in the height order.
	For each $i\in[n]$, let
	\[
	d_i \ := \ h(v_i)-i,
	\qquad
	M \ := \ \max_{i\in[n]}|d_i|.
	\]
	Note that in  the Fibonacci example, we have
	\[
	\big|h(x_i)-(m+i+1)\big|
	\quad = \quad
	\frac{F_{2|i|}}{F_{2m+2}}.
	\]
	{Thus the displacement at every fixed index tends to zero as $m\to\infty$}, while its maximum absolute value is attained at the boundary and remains bounded away from zero.
	This suggests that the displacement function may help locate the better-balanced pairs that are lost in the infinite-volume limit.
	{The next two lemmas make} this observation rigorous.

	\subsection{Two key lemmas for the proof of Theorem~\ref{thm:ACPP}}\label{ss:balance-displacement}

	\smallskip
	
	\begin{lemma}[{\rm \defna{Choosing a BFT triple}}{}]\label{lem:extremal-BFT}
		Suppose that $G({P})$ is connected and $n\geq3$.
		{Then there exists} $k\in[n-2]$ such that $(v_k,v_{k+1},v_{k+2})$ is a BFT triple and
		\[
		\max\big\{|d_k|,\,|d_{k+2}|\big\} \quad = \quad M.
		\]
	\end{lemma}
	
	\smallskip
	
	\begin{proof}
		Since the positions in a linear extension are distinct integers in $[n]$, we have
		\[
		d_1\geq0, \qquad d_n\leq0,
		\qquad
		d_1+d_2\geq0, \qquad d_{n-1}+d_n\leq0.
		\]
		It follows that either $d_i=M$ for some $i\leq n-2$, or $d_j=-M$ for some $j\geq3$.
		Set $k=i$ in the first case and $k=j-2$ in the second.
		In either case,
		\[
		h(v_{k+2})-h(v_k)
		\quad = \quad 2+d_{k+2}-d_k
		\quad \leq \quad 2,
		\]
		and the displacement of one of the two endpoints has absolute value $M$.
		
		It remains to show that these three elements do not form a chain.
		Write \. $x:=v_k$, \. $y:=v_{k+1}$, \. and \. $z:=v_{k+2}$ \. and suppose that they form a chain. Then $x\prec y\prec z$, so every linear extension $g$ satisfies $g(z)-g(x)\geq2$.
		The bound on the mean span would therefore force $g(z)-g(x)=2$ in every linear extension.
		Since $G(P)$ is connected, there exists an element $w$ which is incomparable with $y$.
		Then neither $w\prec x$ nor $z\prec w$ can hold.
		So  there exists  a linear extension $g$ with $g(x)<g(w)<g(z)$.
		This implies that both $w$ and $y$ would lie between $x$ and $z$, giving $g(z)-g(x)\geq3$, a contradiction.
	\end{proof}
	
	\smallskip
	
	Fix the triple $(x,y,z)$ from Lemma~\ref{lem:extremal-BFT}.
	The next estimate uses the structural conclusions of Lemma~\ref{lem:rigidity} to obtain a lower bound for the balance of this triple in terms of $M$.
	
	\smallskip
	
	\begin{lemma}\label{lem:balance-displacement}
		Let ${P}$ be a finite poset satisfying
		\[
		\pi({P}) \ \leq \  D
		\qquad \text{ and } \qquad
		\delta({P}) \ \leq \ \BFT+\eta_D,
		\]
		and let $(x,y,z)$ be a BFT triple  chosen as in Lemma~\ref{lem:extremal-BFT}.
		Then
		\begin{equation}\label{eq:balance-displacement}
			\max\big\{\delta({P};x,y),\,\delta({P};y,z)\big\} \quad \geq \quad \BFT+\frac{M}{6}.
		\end{equation}
	\end{lemma}
	
	\smallskip

	\begin{proof}
		Let $p,p'$ be given by 
		\begin{equation*}
			p \ := \ \Pb[\cL(y)<\cL(x)],
			\qquad
			p' \ := \ \Pb[\cL(z)<\cL(y)].
		\end{equation*}
		It follows from Corollary~\ref{lem:comp} that 
		\[
		\delta(P;x,y)=p,
		\qquad
		\delta(P;y,z)=p'.
		\]
		
		Now, by Lemma~\ref{lem:rigidity}, the triple $(x,y,z)$ satisfies all
		the structural conclusions stated there.
		Let $L,U$ be as in that lemma, and let  \. $m:=|L|+1$\..
		Since $L\prec y\prec U$ and $x,y,z$ are consecutive in the height order, every element of $L$ occurs before $x$ in that order, and every element of $U$ occurs after $z$.
		Thus   \. $x=v_m$, \. $y=v_{m+1}$, \. and \. $z=v_{m+2}$\..
		
		We now  write 
		\[
		\alpha \ := \ 
		\Eb\bigl[|\{w\in L:\cL(x)<\cL(w)\}|\bigr],
		\qquad
		\beta \ := \ 
		\Eb\bigl[|\{w\in U:\cL(w)<\cL(z)\}|\bigr].
		\]
		Counting the elements before $x$ and $z$ now  gives
		\[
		h(x) \ = \ m+p-\alpha,
		\qquad
		h(z)\ = \  m+2+\beta-p',
		\]
		and hence
		\begin{equation}\label{eq:bridge-endpoint-displacements}
			d_m=p-\alpha,
			\qquad
			d_{m+2}=\beta-p'.
		\end{equation}
		The  assumption \.  $h(z)-h(x)\le2$  \. therefore gives
		\begin{equation}\label{eq:bridge-displacement-order}
			d_{m+2} \ \le  \ d_m.
		\end{equation}
		Now, let 
		\[
		a:=\max\{p,p'\}
		=\max\bigl\{\delta(P;x,y),\,\delta(P;y,z)\bigr\}.
		\]
		Note that our assumption implies 
		\. 
		$a\le\delta(P)\le\BFT+\eta_D<\frac13$\..
		For $t \in [0,1/3)$, 
		let  $V(t)$  be given by 	\[
		V(t):=\frac{(1-t)(1-3t)}{1-2t}.
		\]
		The following claim gives an upper and lower bound for \eqref{eq:bridge-displacement-order} in terms of $a$ and $V(a)$.
		
		\smallskip
		
		\begin{claim}We have \label{lem:bridge-displacement-squeeze}
			\begin{equation*}
				V(a)-a
				\quad\le\quad
				d_{m+2}
				\quad\le\quad
				d_m
				\quad\le\quad
				a-V(a).
			\end{equation*}
		\end{claim}
		
		\smallskip
		
		\begin{proof}
			Let 
			\[
			L':=L\cup\{x\},
			\qquad
			U':=U\cup\{z\}.
			\]
			Let $\cL_{L'}$ and $\cL_{U'}$ be independent uniform linear
			extensions of $P[L']$ and $P[U']$, respectively. 
			Set
			\[
			r \ := \ \Pb[\cL_{L'}(x)=m],
			\qquad
			s \ := \ \Pb[\cL_{U'}(z)=1].
			\]
			By~\eqref{eq:bridge-order}, we have
			\[
			L' \quad \prec \quad U'.
			\]
			Note that
			each pair of linear extensions of ${P}[L']$ and ${P}[U']$ gives exactly one extension of ${P}$ in which $x$ precedes $y$ and $y$ precedes $z$: concatenate the two extensions and insert $y$ between them.
			There is one additional extension with $y$ preceding $x$ precisely when $x$ is last in the left extension, and one with $z$ preceding $y$ precisely when $z$ is first in the right extension.
			These three possibilities exhaust the linear extensions of ${P}$.
			Consequently,
			\begin{equation*}
				e({P}) \quad = \quad e({P}[L']) \. e({P}[U']) \. (1+r+s),
			\end{equation*}
			which implies that 
			\begin{equation*}
				p \ = \ \Pb[\cL(y)<\cL(x)] \ = \ \frac{r}{1+r+s},
				\qquad
				p' \ = \ \Pb[\cL(z)<\cL(y)] \ = \ \frac{s}{1+r+s}.
			\end{equation*}
			Since $1-p-p'=\frac{1}{1+r+s}$, we have
			\[
			r \ = \ \frac{p}{1-p-p'}
			\ \le \ \frac{a}{1-2a},
			\qquad
			s \ = \ \frac{p'}{1-p-p'}
			\le\frac{a}{1-2a}.
			\]

			We now give a lower bound for $\alpha$ in terms of $r,s$. 
			First note that it follows from the definition of $\alpha$ that 
			\[
			\alpha
			\quad\ge\quad
			\Pb\bigl[ \exists \. w \in L  \. : \.  \cL(x) < \cL(w)\bigr].
			\]
			Now note that, the existence of \. $w \in L$ \. satisfying  \. $\cL(x) < \cL(w)$\., implies that 
			\. $\cL(x) < \cL(y)$.
			Conditional on the latter event, note that the induced total order on $L'$ by $\cL$ is uniform. 
			Therefore,
			\begin{align*}
				\alpha
				\quad &\ge\quad
				\Pb\bigl[ \exists \. w \in L  \. : \.  \cL(x) < \cL(w)\bigr] 
				\quad  = \quad \Pb[\cL(y)>\cL(x)] \. \Pb[\cL_{L'}(x)<m]   \\
				\quad &=\quad
				(1-p)(1-r).
			\end{align*}
			The analogous argument applied to $\beta$  gives
			\[
			\beta \ \ge \ (1-p')(1-s).
			\]
			It  follows that
			\[
			\alpha,\beta
			\quad\ge\quad
			(1-a)\left(1-\frac{a}{1-2a}\right)
			\quad=\quad
			V(a).
			\]
			Combining these bounds with
			\eqref{eq:bridge-endpoint-displacements}, we obtain
			\begin{align*}
				d_m  \ = \  p- \alpha  \ \leq \ a -V(a), \qquad d_{m+2} \  = \  \beta -p' \ \geq \ V(a)-a.
			\end{align*}
			Together with \eqref{eq:bridge-displacement-order}, this completes the proof of the claim.
		\end{proof}		
		
		\smallskip
		
		We now give an estimate for $a-V(a)$ that follows from an elementary argument. 
		
		\smallskip
		\begin{claim}\label{lem:calculus}
			We have 
			\[  0 \ \leq \   a-V(a) \ \leq \ 6 \. (a -\BFT).  \]
		\end{claim}
		
		\smallskip
		
		\begin{proof}
			Note that it follows directly  from Claim~\ref{lem:bridge-displacement-squeeze} that $a \geq V(a)$. 
			This implies that 
			\. $5a-5a^2-1 \. \ge \. 0$\..
			Since $a<1/3$, this implies $a\ge\BFT$.
			We now have 
			\begin{align*}
				a-V(a)
				\quad = \quad \frac{5a-5a^2-1}{1-2a}
				\quad = \quad  5(a-\BFT)
				\left(1+\frac{a-\BFT}{1-2a}\right).
			\end{align*}
			where the second equality uses   \. $5\BFT-5\BFT^2-1=0$\..
			Now note that, since  $\BFT\le a<1/3$, we then have 
			\[
			\frac{a-\BFT}{1-2a}
			\quad \le \quad  1\ - \ 3 \. \BFT
			\quad < \quad \frac15,
			\]
			This completes the proof.		
		\end{proof}

		\smallskip

		Now, it follows from 
		Claim \ref{lem:bridge-displacement-squeeze} and Claim~\ref{lem:calculus} that 
		\[
		M \ = \ \max\{|d_m|,|d_{m+2}|\}
		\ \le \  a-V(a)
		\ \le \ 6(a-\BFT).
		\]
		Rearranging gives
		\[
		\max\bigl\{\delta(P;x,y),\,\delta(P;y,z)\bigr\}
		\ = \ a \ \ge \ \BFT+\frac{M}{6},
		\]
		as desired.
	\end{proof}

	\medskip

	\subsection{Completion of the proof}\label{ss:complete-small-range}

	Since $G({P})$ is connected, some element $v_j$ is incomparable with $v_1$.
	The event \. $\cL(v_j)<\cL(v_1)$ \. has positive probability and involves only two elements.
	Lemma~\ref{lem:feasible} therefore gives
	\begin{equation}\label{eq:displacement-lower}
		M \quad \geq \quad d_1
		\quad = \quad \sum_{j=2}^{n}\Pb[\cL(v_j)<\cL(v_1)]
		\quad \geq \quad q_2.
	\end{equation}
	({Note that finiteness is \emph{important} here. It guarantees the existence of 
		a \emph{first element} $v_1$ in the height order, whose displacement
		$d_1=h(v_1)-1=\sum_{j=2}^{n}\Pb[\cL(v_j)<\cL(v_1)]$
		is a sum of nonnegative terms. The displacement of a general element,
		however, may involve both positive and negative contributions.
		Such a first element need not exist in an infinite poset, as illustrated
		by the infinite Fibonacci poset.})

	We may  assume
	\. $\delta({P}) \  \leq \ \BFT+\eta_D$, as otherwise the result is immediate.
	Choose $(x,y,z)$ as in Lemma~\ref{lem:extremal-BFT}.
	Lemma~\ref{lem:balance-displacement}
	and~\eqref{eq:displacement-lower} therefore give
	\begin{align*}
		{\delta({P})}
		&{\ \geq\ \max\big\{\delta({P};x,y),\,\delta({P};y,z)\big\}}\\
		&{\ \geq\ \BFT+\frac{M}{6}
			\ \geq\ \BFT+\frac{q_2}{6}\ > \  \BFT+\frac{q_3}{12}  \ = \   \BFT+\eta_D,}
	\end{align*}
	which proves the theorem. \qed

	\smallskip

	\begin{rem}
		In fact, our proof yields the slightly stronger bound  for  Theorem~\ref{thm:ACPP}: 		\[ \delta(P) \ > \ \frac{5-\sqrt{5}}{10} \. + \. \frac{(D+1)^{-3(D+1)}}{5+3\sqrt{5}},  \]
		where \. $5+3\sqrt{5}\approx11.7082<12$\..
		We retain the denominator $12$ in Theorem~\ref{thm:ACPP}
		for aesthetic reasons and to simplify some steps of the proof.
	\end{rem}
	
	
	\medskip

	\section{Preliminaries for  Theorem~\ref{thm:Aires}}\label{s:prelim-Aires}
	
	In this section we review results in the literature that will be used in the proof of Theorem~\ref{thm:Aires}.
	
	\subsection{The order polytope and log-concave random variables}\label{ss:width-probability}
	
	Let ${P}=(X,\prec)$ be a finite poset with $n$ elements.
	The \defnb{order polytope} of ${P}$, introduced by Stanley~\cite{Sta-AF,Sta-two}, is
	\[
	\cO({P}) \quad := \quad
	\bigl\{(z_x)_{x\in X}\in[0,1]^X:
	z_x\leq z_y \text{ whenever } x\prec y\bigr\}.
	\]
	Let \. $\mathbf{F}=(F_x)_{x\in X}$ \. be chosen uniformly from $\cO({P})$.
	The random point $\mathbf{F}$ can be coupled  with a uniformly random linear extension $\cL$ of ${P}$
	by requiring that, almost surely,
	\[
	F_x < F_y
	\quad \Longleftrightarrow \quad
	\cL(x) < \cL(y)
	\qquad (x,y\in X).
	\]
	Note that
	the induced linear extension is uniform because the simplices corresponding
	to the linear extensions of ${P}$ have equal volume.
	In particular, the comparisons among the coordinates of $\mathbf{F}$ have the same probabilities as the corresponding comparisons in $\cL$.
	Throughout this section and the next, we use the notation
	\[
	Z_x \ := \ (n+1)F_x,
	\qquad
	\Delta(x,y) \ := \ |h(x)-h(y)|,
	\qquad
	\dev(x,y) \ := \  \sqrt{\operatorname{Var}(Z_x-Z_y)}.
	\]
	Note that we have \. $\Eb[Z_x]=\Eb[\cL(x)]=h(x)$\.,
	the height function of  ${P}$.
	We will use the following folklore result that is a consequence of the Brunn--Minkowski inequality, and for which the earliest written
	reference we have found is~\cite[Cor.~5.1(a)]{AK25a}.
	
	\smallskip
	
	\begin{lemma}[{\rm \defna{Log-concave density}~{\cite[Cor.~5.1(a)]{AK25a}}}{}]\label{lem:lc-density}
		For distinct $x,y \in X$,
		the random variable $Z_x-Z_y$  has a nondegenerate log-concave density.
	\end{lemma}

	\smallskip

	{The following two elementary estimates follow from the results of Lov\'asz and Vempala~\cite[Lems.~5.4, 5.5(a), and~5.7]{LV07}.}
	We include a short derivation here, as our formulation differs slightly
	from the cited statements.

	\begin{lemma}[{\rm \defna{Log-concave estimates}}{}]\label{lem:width-logconcave}
		For every random variable $V$ with a log-concave density, mean $\mu$, and standard deviation $\sigma>0$,
		\begin{equation}\label{eq:width-logconcave-balance}
			\min\bigl\{\Pb[V < 0],\,\Pb[V > 0]\bigr\}
			\quad \geq \quad e^{-1}-\frac{|\mu|}{\sigma},
		\end{equation}
		and
		\begin{equation}\label{eq:width-logconcave-tail}
			\Pb[|V-\mu|>t\sigma]
			\quad \leq \quad  e^{-t+1}
			\qquad (t\geq0).
		\end{equation}
	\end{lemma}
	
	\smallskip
	
	\begin{proof}
		For  \eqref{eq:width-logconcave-balance},  set $W=(V-\mu)/\sigma$, and let $g$ be its density function.
		Applying {\cite[Lem.~5.4]{LV07}} to $W$ and $-W$ gives
		\[
		\min\{\Pb[W < 0],\,\Pb[W > 0]\}  \quad = \quad \min\{\Pb[W \leq  0],\,\Pb[W \geq  0]\} \quad \geq \quad e^{-1}.
		\]
		(Since $W$ has a density, it has no atoms, so strict and weak
		inequalities in the probability events are interchangeable.)
		Since $W$ has mean zero and variance one, {\cite[Lem.~5.5(a)]{LV07}} gives
		$g\leq1$. Thus moving the threshold from $0$ to $-\mu/\sigma$
		changes either probability by at most
		\[
		\int_{\min\{0,-\mu/\sigma\}}^{\max\{0,-\mu/\sigma\}}g(s)\,ds
		\quad \leq \quad  \frac{|\mu|}{\sigma}.
		\]
		This implies
		\[
		\min\{\Pb[V < 0],\,\Pb[V > 0]\}
		\quad \geq \quad  e^{-1}-\frac{|\mu|}{\sigma},
		\]
		as desired.
		
		For \eqref{eq:width-logconcave-tail},
		we apply~\cite[Lem.~5.7]{LV07} to  \. $(V-\mu)/\sigma$ \.  to get
		the stated conclusion for  $t>1$.
		For the case $t \in [0,1]$, the bound
		is trivial since $e^{1-t}\geq1$.
	\end{proof}

	\subsection{Three auxiliary inequalities}\label{ss:width-inequalities}
	
	For each $x\in X$, let
	\[
	Q_x \ := \ \max_{y\prec x}F_y,
	\qquad
	R_x \ := \ \min_{x\prec y}F_y\textcolor{red}{,}
	\]
	where an empty maximum is $0$ and an empty minimum is $1$.
	Let
	\[a_x \ := \ (n+1)\Eb[F_x-Q_x] \ = \ \frac{n+1}{2}\Eb[R_x-Q_x],\]
	where the second equality follows because, conditional on all coordinates
	other than $F_x$, the random variable $F_x$ is uniform on $[Q_x,R_x]$,
	so $F_x-Q_x$ has conditional mean $(R_x-Q_x)/2$.
	The number $2a_x$ is known as  the \defnb{discrete window parameter} $\operatorname{win}(x)$
	in the work of  Aires and Kahn~\cite[(6)]{AK25a}.

	The following lemma follows from the estimates in \cite[Eq.~(24), Thm.~2.9]{AK25a}.
	We include a short derivation here, as our formulation differs slightly
	from the cited statements.
	
	\begin{lemma}[{\rm \defna{Window estimates}}{}]\label{lem:width-windows}
		For distinct $x,y\in X$ and every antichain $B\subseteq X$,
		\begin{equation*}
			a_x \ \geq \ 1,
			\qquad
			\dev(x,y) \ \geq \ \frac{a_x}{\sqrt{3}},
			\qquad
			\sum_{x\in B}a_x \ \geq \ \frac{|B|^2}{2}.
		\end{equation*}
	\end{lemma}

	\begin{proof}
		The coupling between ${\mathbf F}$ and $\cL$ gives
		\[
		a_x
		\ = \
		\Eb\bigl[\cL(x)-\max_{y\prec x}\cL(y)\bigr]
		\ \geq  \ 1,
		\]
		where the empty maximum is $0$.
		This proves the first claim.
		For the second claim, we use the same conditioning argument applied in \cite[Eq.~(24)]{AK25a}.
		Let $\mathcal G=\sigma(F_z:z\neq x)$.
		Conditional on $\mathcal G$, the variable
		$F_x$ is uniform on $[Q_x,R_x]$, and hence
		%
		\[
		\operatorname{Var}(Z_x-Z_y\mid\mathcal G)
		=(n+1)^2\operatorname{Var}(F_x\mid\mathcal G)
		=\frac{(n+1)^2}{12}(R_x-Q_x)^2.
		\]
		The law of total variance now gives
		\begin{align*}
			\dev(x,y)^2 \ = \ 	\operatorname{Var}(Z_x-Z_y)
			&=\Eb\bigl[\operatorname{Var}(Z_x-Z_y\mid\mathcal G)\bigr]
			+\operatorname{Var}\bigl(\Eb[Z_x-Z_y\mid\mathcal G]\bigr)\\
			&\geq \Eb\bigl[\operatorname{Var}(Z_x-Z_y\mid\mathcal G)\bigr]\\
			&=\frac{(n+1)^2}{12}\Eb[(R_x-Q_x)^2]  \ \geq  \ \frac{a_x^2}{3},
		\end{align*}
		where the last inequality is because of Jensen's inequality.
		For the third claim, notice that 	
		$a_x=\operatorname{win}(x)/2$ in the notation
		of~\cite{AK25a}. Thus~\cite[Thm.~2.9]{AK25a} gives
		\[
		\sum_{x\in B} a_x
		\ 	\geq  \   \frac{n+1}{2n}|B|^2
		\ 	\geq \  \frac{|B|^2}{2},
		\]
		as desired.
	\end{proof}
	\smallskip

	The following lemma is a consequence of the continuous form of Shepp's XYZ inequality~\cite[Sec.~2]{She82},
	which in turn is a consequence of the classical FKG inequality~\cite{FKG71}.
	
	\smallskip
	\begin{lemma}[{\rm \defna{Correlation inequality}}{}]\label{lem:width-correlations}
		For all $x,y,z\in X$, we have
		\begin{equation*}
			\dev(x,z)^2 \quad \leq \quad \dev(x,y)^2 \ + \ \dev(y,z)^2.
		\end{equation*}
	\end{lemma}
	
	\begin{proof}
		The continuous form of   Shepp's XYZ inequality~\cite{She82}
		says that  the random variables \. $Z_x-Z_y$ \. and  \. $Z_z-Z_y$ \.  are positively associated.
		In particular, they have  {nonnegative covariance} and hence
		\begin{align*}
			\dev(x,z)^2
			&\ = \ \operatorname{Var}\bigl((Z_x-Z_y)-(Z_z-Z_y)\bigr) \\
			&\ \leq \ \operatorname{Var}(Z_x-Z_y)+\operatorname{Var}(Z_z-Z_y)
			\ = \ \dev(x,y)^2+\dev(y,z)^2,
		\end{align*}
		as desired.
	\end{proof}
	
	\smallskip

	The next lemma follows from the \defnb{Haqi--Kahn ideal inequality},
	formerly known as Kahn's ideal conjecture, which was recently proved
	by Haqi~\cite[Thm.~4.1]{Haq26}.
	An \defnb{order ideal} is a subset $I\subseteq X$ such that \. $x\in I$ \. and \. $y\prec x$ \. imply \. $y\in I$\..
	We write \.  $\max(I):= \max_{P}(I)$ \.  and  \. $\min(J):=\min_{P}(J)$ \.  {for the sets of maximal elements of ${P}[I]$ and minimal elements of ${P}[J]$, respectively.}
	
	\smallskip
	
	\begin{lemma}\label{lem:width-ideal-gap}
		Let $I$ be a nonempty proper order ideal of ${P}$, and let \. $J:=X\setminus I$\..
		Then
		\begin{equation*}
			\min_{y\in J}h(y)-\max_{x\in I}h(x)
			\quad \leq \quad
			|\max(I)|+|\min(J)|-1.
		\end{equation*}
	\end{lemma}
	
	\begin{proof}
		Haqi--Kahn ideal inequality~\cite[Thm.~4.1]{Haq26} gives
		\[
		\max_{x\in I}h(x)
		\quad \geq \quad |I|-|\max(I)|+1.
		\]
		Applying the same inequality to $J$ in the dual poset  gives
		\[
		\min_{y\in J}h(y) \ =  \  n+1- \max_{y\in J} h^*(y)
		\  \leq  \    n+1 -(|J|-|\min(J)|+1)   \ = \ |I|+|\min(J)|,
		\]
		where $h^*$ is the height function in the dual poset.
		{Taking the difference of the two inequalities proves the result.}
	\end{proof}
	
	\medskip
	
	\section{Proof of Theorem~\ref{thm:Aires}}\label{s:proof:Aires}
	
	Throughout this section, we write 
	\[ \tau \ := \  10^{-100},  \]
	and we assume that the width of $P$ satisfies
	\[   w(P) \ > \   	\frac{18432\times 3061260}{\tau^4}.  \]
	We will work toward showing  that \. $\delta(P) >  e^{-1}-\tau$, thereby proving Theorem~\ref{thm:Aires}. 
	
	\smallskip

	We may assume without loss of generality that ${P}$ has a unique minimum element
	and a unique maximum element, since adjoining these elements changes neither
	$\delta({P})$ nor $w({P})$.
	We  list the elements of ${P}$ as $v_1,\ldots,v_n$ in {nondecreasing}
	order of height,
	\[
	h(v_1) \quad {\leq} \quad \cdots \quad {\leq} \quad h(v_n),
	\]
	and write
	\[
	\gap({P}) \quad := \quad
	\max\bigl\{h(v_1),\ h(v_{i+1})-h(v_i)\ (1\leq i<n),\
	n+1-h(v_n)\bigr\}.
	\]
	Since ${P}$ has a minimum element and a maximum element, we have
	\. $h(v_1)=n+1-h(v_n)=1$ \. and so \. $\gap({P})\geq1$\..
	The proof of Theorem~\ref{thm:Aires} uses  the following two different lower bounds for $\delta(P)$, both of them being consequences of log-concave estimates of  Lov\'asz--Vempala~(Lemma~\ref{lem:width-logconcave}).
	
	\smallskip

	\begin{prop}[{\rm \defna{First lower bound for $\delta(P)$}}{}]\label{prop:small-gap}
		There exist distinct $x,y \in X$ such that 
		\[ \delta(P;x,y)  \quad  \geq \quad e^{-1} \ - \  \frac{2\sqrt{3} \.  \gap(P)}{w(P)}.  \]
	\end{prop}
	
	\smallskip
	
	We will prove Proposition~\ref{prop:small-gap} in Section~\ref{ss:small-gap}.
	Note that Theorem~\ref{thm:Aires} immediately follows from the proposition if $\gap(P)$ is relatively \emph{small}, i.e. if   \.  $\gap(P) \ < \  \tfrac{\tau}{2\sqrt{3}} \. w(P)$\.. 
	Thus  we will assume for the rest of the proof that  \. $\gap(P)$ is relatively \emph{big}, i.e.
	\begin{equation}\label{eq:gap-bound-1}
		\gap(P) \ \geq \  \frac{\tau}{2\sqrt{3}} \. w(P)  \ > \ \frac{(18432 \times 510210) \sqrt{3}}{\tau^3}.
	\end{equation}

	\smallskip
	
	The second lower bound for $\delta(P)$ requires the following notation.
	For a nonempty subset $B\subseteq X$, let
	\[
	H(B) \quad := \quad \max_{x\in B}h(x) \ - \ \min_{x\in B}h(x),
	\]
	the \defnb{{span of the average heights}} in $B$.
	
	\smallskip
	
	\begin{prop}[{\rm \defna{Second lower bound for $\delta(P)$}}{}]\label{prop:large-gap}
		Let $B \subseteq X$ be an antichain of $P$ with $|B| \geq 2$. Then 
		there exist distinct $x,y \in B$ such that 
		\[ \delta(P;x,y)  \quad  \geq \quad e^{-1} \ - \  \frac{4\sqrt{3} \.  H(B)}{|B|^2}.  \]
	\end{prop}
	
	\smallskip
	We will prove Proposition~\ref{prop:large-gap} in Section~\ref{ss:large-gap}.
	We will apply this proposition to an antichain $B$  that satisfies the following properties.
	For the rest of the proof we will further assume  that 
	\begin{equation}\label{eq:width-deficit}
		\delta({P}) \quad \leq \quad e^{-1}-\tau,
	\end{equation}
	as otherwise Theorem~\ref{thm:Aires} would follow immediately.		
	
	\smallskip

	\begin{prop}[{\rm \defna{Existence of nice antichains}}{}]\label{prop:canopy-antichains}
		Under assumptions \eqref{eq:gap-bound-1} and \eqref{eq:width-deficit}, there exists an antichain $B$ satisfying 
		\[ |B|  \  \geq \ \frac{\gap(P)}{4} \quad \text{ and } \quad H(B) \ \leq \  \frac{288}{\tau^2} \gap(P) \. (\log \gap(P))^2.  \]
	\end{prop}
	
	\smallskip
	
	We will prove Proposition~\ref{prop:canopy-antichains} in Section~\ref{ss:canopy-antichains}.
	Combining Propositions~\ref{prop:large-gap} and \ref{prop:canopy-antichains}, we obtain 
	\begin{align*}
		\delta(P) \quad \geq \quad e^{-1} \ - \  \frac{18432\sqrt{3}}{\tau^2} \. \frac{(\log \gap(P))^2}{\gap(P)}  \quad &\geq \quad e^{-1} \ - \  \frac{\tau}{510210}
		\left(
		\log\frac{(18432\cdot510210)\sqrt{3}}{\tau^3}
		\right)^2 \\
		\quad &\approx  \quad    e^{-1} \ - \  0.999998 \. \tau \quad > \quad e^{-1} \ - \ \tau,
	\end{align*}
	where the second inequality follows from \eqref{eq:gap-bound-1}.
	%
	We have thus proved Theorem~\ref{thm:Aires} modulo the three propositions, whose proofs are given in the remainder of this section.

	\medskip

	\subsection{Proof of Proposition~\ref{prop:small-gap}}\label{ss:small-gap}
	Let $B\subseteq X$ be an antichain of size $w(P)$.
	By the third part of  Lemma~\ref{lem:width-windows}, we have
	\[
	\sum_{z\in B}a_z \quad \geq \quad \frac{w(P)^2}{2}.
	\]
	Thus there exists $x\in B$ such that
	\[
	a_x \quad \geq \quad \frac{w(P)}{2}.
	\]
	Choose $y$ adjacent to $x$ in the height order
	$v_1,\ldots,v_n$. By the definition of $\gap(P)$, we have
	\[
	\Delta(x,y) \quad \leq \quad \gap(P).
	\]
	On the other hand, 
	the second part  of Lemma~\ref{lem:width-windows}  gives
	\[
	\dev(x,y)
	\quad \geq \quad \frac{a_x}{\sqrt{3}}
	\quad \geq \quad \frac{w(P)}{2\sqrt{3}}.
	\]
	By Lemma~\ref{lem:lc-density}, the random variable $Z_x-Z_y$
	has a log-concave density, with mean $h(x)-h(y)$ and
	standard deviation $\dev(x,y)$.
	Applying Lemma~\ref{lem:width-logconcave}, we then  obtain
	\[
	\delta(P;x,y)
	\quad \geq \quad e^{-1}-\frac{\Delta(x,y)}{\dev(x,y)}
	\quad \geq \quad e^{-1}
	-\frac{2\sqrt{3}\,\gap(P)}{w(P)},
	\]
	as desired. \qed
	
	\medskip
	
	\subsection{Proof of Proposition~\ref{prop:large-gap}}\label{ss:large-gap}	
	
	Write $b:=|B|$, and list the elements of $B$ as
	$x_1,\ldots,x_b$ in nondecreasing order of height.
	Let  $y_i:=x_{i+1}$ for $i \in \{1,2,\ldots,b-1\}$, and let
	$y_b:=x_{b-1}$. Then
	\[
	\sum_{i=1}^{b}\Delta(x_i,y_i)
	\quad = \quad H(B)+h(x_b)-h(x_{b-1})
	\quad \leq \quad 2 \. H(B).
	\]
	On the other hand, the third part of  Lemma~\ref{lem:width-windows} gives
	\[
	\sum_{i=1}^{b}a_{x_i}
	\quad \geq \quad \frac{b^2}{2}.
	\]
	Since  $a_{x_i}$ are positive by the first part of  Lemma~\ref{lem:width-windows}, there exists an index $j$ such that
	\begin{equation}\label{eq:vesna}
		\frac{\Delta(x_j,y_j)}{a_{x_j}}
		\quad \leq \quad
		\frac{\sum_{i=1}^{b}\Delta(x_i,y_i)}
		{\sum_{i=1}^{b}a_{x_i}}
		\quad \leq \quad \frac{4H(B)}{b^2}.
	\end{equation}
	Set $x:=x_j$ and $y:=y_j$, which are distinct elements of $B$ by the definition of $y_j$'s.
	By Lemma~\ref{lem:lc-density}, the random variable $Z_x-Z_y$
	has a log-concave density, with mean $h(x)-h(y)$ and
	standard deviation $\dev(x,y)$.
	We then obtain
	\begin{align*}
		\delta(P;x,y)
		\quad 	&\geq \quad  e^{-1}-\frac{\Delta(x,y)}{\dev(x,y)} \quad \geq \quad  e^{-1}-\frac{\sqrt{3}\,\Delta(x,y)}{a_x}\\
		&\geq e^{-1}-\frac{4\sqrt{3}\,H(B)}{|B|^2},
	\end{align*}
	where the first inequality is due to \eqref{eq:width-logconcave-balance}, the 
	second inequality is due to the second part of
	Lemma~\ref{lem:width-windows},
	and the third inequality is due to \eqref{eq:vesna}.
	This completes the proof.
	\qed 
	
	\medskip

	\subsection{Proof of Proposition~\ref{prop:canopy-antichains}}\label{ss:canopy-antichains}	
	For the rest of this section, we retain assumptions
	\eqref{eq:gap-bound-1} and \eqref{eq:width-deficit}.
	Our goal is to find an antichain $B$ with large size $|B|$ and with small span $H(B)$.
	We first explain how the  antichain is constructed, and then prove
	the estimates needed to bound the two given parameters.
	
	Since \eqref{eq:gap-bound-1} implies $\gap(P)>1$, there are
	consecutive elements $u,v$ in the height order such that
	\. 
	$h(v)-h(u) \. = \. \gap(P)$\..
	Set
	\[
	I \ :=\ \{x\in X:h(x)\leq h(u)\},
	\qquad
	J \ :=\ X\setminus I \ = \ \{y\in X:h(y)\geq h(v)\}.
	\]
	Note that  $I$ is a nonempty proper order ideal of $P$.
	The natural candidates for the antichain $B$  are the two antichains of the following form
	\begin{equation}\label{eq:beautiful-set}
		\max(I)
		\qquad\text{and}\qquad
		\min(J),
	\end{equation}
	where $\min(S)$ and $\max(S)$ denote the sets of
	minimal and maximal elements of the induced poset ${P}[S]$,
	respectively.
	We call the sets in~\eqref{eq:beautiful-set} \defnb{canopy antichains}.\footnote{%
		The name comes from viewing these antichains as canopies over the
		corresponding order ideals in the poset or its dual.}	{There are two reasons for choosing sets of this form.}

	The first reason comes from the
	{Haqi--Kahn ideal inequality}~(Lemma~\ref{lem:width-ideal-gap}),
	which gives
	\[
	|\max(I)|+|\min(J)|-1
	\quad \geq \quad
	h(v)-h(u)
	\quad = \quad
	\gap(P).
	\]
	Thus at least one of these antichains, say $\max(I)$,  has size at least
	$\gap(P)/2$, which is twice the lower bound required in
	Proposition~\ref{prop:canopy-antichains}.
	The second reason is a bit of \emph{wishful thinking}:
	when $\gap(P)$ is very large, say close to $n$, one would  expect
	the average heights of elements in   $\max(I)$ to  \defna{cluster near $h(u)$}, 
	giving the  antichain a small height span.
	One might hope that this \emph{remains true} whenever $\gap(P)$ is large,
	even when it is much smaller than $n$.
	Unfortunately, we cannot prove this directly, so we instead
	\defna{force this clustering by adding relations}.
	These relations exclude elements whose average heights lie far below
	$h(u)$ from the canopy.
	We will show that this construction produces a canopy antichain
	of size at least $\gap(P)/4$. Thus we lose only a factor of $2$
	in the lower bound, which  still meets  the requirement of
	Proposition~\ref{prop:canopy-antichains}.
	
	\smallskip
	
	We now build toward the proof of Proposition~\ref{prop:canopy-antichains}.
	Set
	\begin{equation}\label{eq:ell}
		\ell \ := \ \frac{288}{\tau^2} \. \gap({P}) \. (\log\gap({P}))^2.
	\end{equation}
	We want the elements in the lower canopy to have their heights in
	$(h(u)-\ell,h(u)]$, and the elements in the upper canopy to have their 
	heights in $[h(v),h(v)+\ell)$.
	Thus, let $\cQ$ be obtained from $P$ by adding
	the comparisons
	\begin{equation}\label{eq:width-added-comparisons}
		z\prec_{\cQ}u
		\quad \text{if } \quad h(z)\leq h(u)-\ell,
		\qquad \text{ and } \qquad 
		v\prec_{\cQ}z
		\quad \text{if } \quad  h(z)\geq h(v)+\ell,
	\end{equation}
	and taking the transitive closure.
	Note that 
	every added comparison is oriented in increasing order
	of the  heights in $P$, so these comparisons are consistent with $\prec_P$ and thus 
	$\cQ$ is a poset.
	Moreover, each added comparison is between two elements
	of $I$ or between two elements of $J$, so $I$ remains
	an order ideal of $\cQ$.
	
	From now on, we write $h_P$ and $h_{\cQ}$ to distinguish
	the two height functions.
	The quantities $\Delta$, $\dev$, $a_x$, and $H(B)$ continue
	to refer to the original poset $P$.
	Define
	\[
	B_0 \ :=\ \max_{\cQ}(I),
	\qquad
	B_1 \ :=\ \min_{\cQ}(J),
	\]
	and let $B$ be the larger of $B_0$ and $B_1$.
	Both sets are nonempty antichains of $\cQ$, and hence
	also antichains of $P$.
	The desired  bound  for $H(B)$ now follows directly from
	the construction.
	Indeed, every $z\in I$ with \. $h_P(z)\leq h_P(u)-\ell$ \.
	lies below $u$ in $\cQ$, and therefore cannot belong to $B_0$.
	The analogous observation applies to $B_1$, giving
	\begin{align*}
		h_P(u)-\ell
		&\ <\ h_P(z)\ \leq\ h_P(u)
		&&\text{for all }z\in B_0,\\
		h_P(v)
		&\ \leq\ h_P(z)\ <\ h_P(v)+\ell
		&&\text{for all }z\in B_1.
	\end{align*}
	Consequently,
	\begin{equation}\label{eq:width-canopy-span}
		H(B_0)\leq\ell
		\qquad \text{ and } \qquad 
		H(B_1)\leq\ell.
	\end{equation}
	It remains to prove that $|B|\geq\gap(P)/4$.
	This requires  the following lemma.

	\smallskip
	
	\begin{lemma}\label{lem:width-preserved-gap}
		Under assumptions \eqref{eq:gap-bound-1} and \eqref{eq:width-deficit},
		\begin{equation*}
			\min_{t\in J}h_{\cQ}(t)-\max_{s\in I}h_{\cQ}(s)
			\quad \geq \quad
			\frac{\gap(P)}{2}.
		\end{equation*}
	\end{lemma}
	
	\smallskip
	
	We remark that the constant $1/2$ on the right-hand side of
	Lemma~\ref{lem:width-preserved-gap} can be replaced by a constant
	much closer to $1$, but the stated bound suffices for our purposes.
	We briefly postpone the proof of Lemma~\ref{lem:width-preserved-gap}
	to show how it completes the proof of Proposition~\ref{prop:canopy-antichains}.
	We return to the proof of the lemma immediately afterward.
	
	\begin{proof}[Proof of Proposition~\ref{prop:canopy-antichains}]
		Since $I$ is a nonempty proper order ideal of $\cQ$,
		the Haqi--Kahn ideal inequality (Lemma~\ref{lem:width-ideal-gap}) applied in $\cQ$ (rather than in $P$), gives
		\[
		|B_0|+|B_1|-1
		\quad \geq \quad
		\min_{t\in J}h_{\cQ}(t)-\max_{s\in I}h_{\cQ}(s)
		\quad \geq \quad
		\frac{\gap(P)}{2},
		\]
		where the last inequality is
		due to Lemma~\ref{lem:width-preserved-gap}.
		It then follows that
		\[
		|B|
		\quad \geq \quad
		\frac{|B_0|+|B_1|}{2}
		\quad \geq \quad
		\frac{\gap(P)}{4}.
		\]
		Also, $B$ is an antichain of $P$, and
		\eqref{eq:width-canopy-span} gives
		\[
		H(B)
		\quad \leq \quad
		\ell
		\quad = \quad
		\frac{288}{\tau^2}\gap(P)(\log\gap(P))^2.
		\]
		This proves the proposition.
	\end{proof}
	
	\smallskip
	
	It remains only to prove Lemma~\ref{lem:width-preserved-gap},
	whose proof is given in the next subsection.
	
	\medskip
	
	\subsection{Proof of Lemma~\ref{lem:width-preserved-gap}}\label{ss:width-preserved-gap}
	The intuition behind Lemma~\ref{lem:width-preserved-gap} is that
	the additional relations defining $\cQ$ already hold with very high
	probability in a uniformly random linear extension of $P$.
	Since passing from $P$ to $\cQ$ amounts to conditioning on these relations,
	we thus expect the differences of average heights to change only slightly.
	Making this intuition rigorous requires the following three technical lemmas.

	\smallskip
	
	\begin{lemma}\label{lem:width-separation}
		Under assumption \eqref{eq:width-deficit}, we have, 
		for all distinct $x,y\in X$,
		\begin{align}
			\label{eq:width-separation-1} \frac{\Delta(x,y)}{\tau} \quad &\geq \quad  \dev(x,y), \quad \text{ and }  \\
			\label{eq:width-separation-2}	 {\Delta(x,y)} \quad &\geq \quad  \frac{\tau}{\sqrt{3}}.
		\end{align}
	\end{lemma}
	
	\begin{proof}
		By Lemma~\ref{lem:lc-density}, the random variable $Z_x-Z_y$
		has a log-concave density, mean $h_P(x)-h_P(y)$, and
		standard deviation $\dev(x,y)$.
		Thus \eqref{eq:width-logconcave-balance} and
		\eqref{eq:width-deficit} give
		\[
		e^{-1}-\frac{\Delta(x,y)}{\dev(x,y)}
		\ \leq\ \delta(P;x,y)
		\ \leq\ \delta(P)
		\ \leq\ e^{-1}-\tau,
		\]
		from  which \eqref{eq:width-separation-1} follows.
		Lemma~\ref{lem:width-windows} then gives
		\[
		1\ \leq\ a_x
		\ \leq\ \sqrt{3}\, \dev(x,y)
		\ \leq\ \frac{\sqrt{3}}{\tau}\Delta(x,y),
		\]
		which proves \eqref{eq:width-separation-2}.
	\end{proof}
	
	\smallskip

	\smallskip
	
	\begin{lemma}\label{lem:width-gap-estimates}
		Under assumption \eqref{eq:width-deficit}, we have, 
		for all $x,y\in X$,
		\begin{equation*}
			\dev(x,y)^2
			\quad \leq \quad
			\frac{\gap(P)\Delta(x,y)}{\tau^2}.
		\end{equation*}
	\end{lemma}
	
	\begin{proof}
		The lemma follows immediately if $x=y$, so we assume that $x$ and $y$ are distinct.
		Let $i,j \in [n]$ be such that $x=v_i$ and $y=v_j$, and without loss of generality we assume $i< j$.
		Repeated applications of Lemma~\ref{lem:width-correlations},
		followed by \eqref{eq:width-separation-1}, give
		\begin{align*}
			&	\dev(v_i,v_j)^2
			\quad \leq \quad  \sum_{k=i}^{j-1} \dev(v_k,v_{k+1})^2 \quad  \leq \quad  \frac{1}{\tau^2}
			\sum_{k=i}^{j-1}
			\bigl(h_P(v_{k+1})-h_P(v_k)\bigr)^2\\[3 pt]
			&\quad  \leq \quad  \frac{\gap(P)}{\tau^2}
			\sum_{k=i}^{j-1}
			\bigl(h_P(v_{k+1})-h_P(v_k)\bigr)
			\quad  = \quad  \frac{\gap(P)}{\tau^2}
			\bigl(h_P(v_j)-h_P(v_i)\bigr),
		\end{align*}
		as desired.
	\end{proof}
	
	\smallskip

	We now estimate the probability that the added comparisons in \eqref{eq:width-added-comparisons}
	fail.
	Recall that $\cO(P)$ is the order polytope of $P$.
	Let $\mathbf{F}$  be a random point of  $\cO(P)$,
	and  
	let   $E$ be the event that all comparisons in
	\eqref{eq:width-added-comparisons} hold in $\mathbf{F}$.

	\smallskip
	
	\begin{lemma}\label{lem:q-small}
		Under assumptions  \eqref{eq:gap-bound-1} and \eqref{eq:width-deficit},
		\begin{equation*}
			\Pb[E^c] \quad \leq \quad \frac{10 \. e \. (1+\sqrt{3})}{\tau^3}\gap({P})^{-5}.
		\end{equation*}
	\end{lemma}
	
	\smallskip
	
	
	\begin{proof}
		Let $x,y$ be any elements satisfying $h_{{P}}(y)>h_{{P}}(x)$.
		Let $V:=Z_y-Z_x$.
		Note that $V$ has mean
		$\Delta(x,y)=h_{P}(y)-h_{P}(x)>0$ and standard deviation $\dev(x,y)$.
		Applying~\eqref{eq:width-logconcave-tail} with
		\. $t=\Delta(x,y)/\dev(x,y)$  \. gives
		\begin{align*}
			\Pb[Z_y<Z_x] \quad &= \quad \Pb[V<0]
			\quad \leq \quad
			\Pb[|V-\Delta(x,y)|>\Delta(x,y)] \\
			&\leq \quad \exp\left(1-\frac{\Delta(x,y)}{\dev(x,y)}\right).
		\end{align*}
		Applying Lemma~\ref{lem:width-gap-estimates}  then gives us
		\begin{equation}\label{eq:width-reversal-tail}
			\Pb[Z_y<Z_x]
			\quad \leq \quad
			\exp\left(1-\tau\sqrt{\frac{\Delta(x,y)}{\gap({P})}}\right).
		\end{equation}

		We now sum this bound over the added comparisons and  group
		them according to their height differences.
		By the union bound,
		\[
		\Pb[E^c] \ \leq \
		\sum_{{h_{{P}}(z)\leq h_{{P}}(u)-\ell}}\Pb[Z_u<Z_z]
		+
		\sum_{{h_{{P}}(z)\geq h_{{P}}(v)+\ell}}\Pb[Z_z<Z_v].
		\]
		We now group the terms in the sums in the following way:
		For the first family, take $(x,y):=(z,u)$; for the second,
		take $(x,y):=(v,z)$.
		Then,
		group the comparisons in each family according to
		\[
		j\gap({P}) \ \leq \ \Delta(x,y)  \ < \  (j+1)\gap({P}).
		\]
		For each fixed $j$, this condition places ${h_{{P}}(z)}$ in an interval
		of length $\gap({P})$.
		Now note that it follows from
		~\eqref{eq:width-separation-2} that  the heights of any two distinct
		elements differ by at least $\tau/\sqrt{3}$.
		This   implies that every interval of length
		$\gap({P})$ contains the heights of at most
		\. $1+(\sqrt{3}/\tau) \. \gap({P})$ \. elements.
		Moreover, \eqref{eq:width-reversal-tail} gives
		\[
		\Pb[Z_y<Z_x]
		\ \leq \
		\exp\left(1-\tau\sqrt{\frac{\Delta(x,y)}{\gap({P})}}\right)
		\ \leq \
		e^{1-\tau\sqrt{j}}.
		\]
		Every added comparison satisfies $\Delta(x,y)\geq\ell$,
		so its interval index is at least
		\[
		j_0 \ :=  \  \left\lfloor\frac{\ell}{\gap({P})}\right\rfloor.
		\]
		Summing over these intervals and accounting for both families,
		we obtain
		\[
		\Pb[E^c] 
		\ \leq \
		2 \. e\left(1+\frac{\sqrt{3}\gap({P})}{\tau}\right)
		\sum_{j\geq j_0}e^{-\tau\sqrt{j}}.
		\]
		Comparison with an integral gives
		\begin{align*}
			\sum_{j\geq j_0}e^{-\tau\sqrt{j}}
			&\quad  \leq \quad  e^{-\tau\sqrt{j_0}}
			+\int_{j_0}^{\infty}e^{-\tau\sqrt{t}}\,dt  \quad  = \quad \left(1+\frac{2\sqrt{j_0}}{\tau}
			+\frac{2}{\tau^2}\right)e^{-\tau\sqrt{j_0}} \\[3 pt]
			& \quad = \quad \frac{e^{-\tau\sqrt{j_0}/2}}{\tau^2}   \left( \tau^2 e^{-\tau \sqrt{j_0}/2} +   2(\tau\sqrt{j_0}+1) \. e^{-\tau \sqrt{j_0}/2}\right) \quad  \leq \quad  \frac{5}{\tau^2} \. e^{-\tau\sqrt{j_0}/2},
		\end{align*}
		where the last inequality uses \.  $\tau<1$  \. and
		\. $(t+1) \. e^{-t/2} \. \leq \. 2$ \. for  \. $t:=\tau \sqrt{j_0} \geq0$\..
		Now note that the definition  of $\ell$ in  \eqref{eq:ell} gives 
		\[
		\sqrt{j_0}
		\quad =  \quad    \sqrt{ \left \lfloor\frac{288}{\tau^2} \.  (\log\gap({P}))^2  \right \rfloor}   \quad \geq \  \quad 
		\sqrt{ \frac{144}{\tau^2} \.  (\log\gap({P}))^2 } \quad \geq \quad \frac{12}{\tau}\log\gap({P}),
		\]
		where we use the assumption that \. $\gap(P)\geq 2$ \.  (which follows from \eqref{eq:gap-bound-1}) in the first inequality.
		Since $\gap({P})\geq1$ and $\tau<1$, it follows that
		\begin{equation*}
			\begin{aligned}
				\Pb[E^c] 
				\quad  \leq \quad  \frac{10 \. e \. (1+\sqrt{3})}{\tau^3}
				\gap({P}) \. e^{-\tau\sqrt{j_0}/2} 
				\quad  \leq \quad  \frac{10 \. e \. (1+\sqrt{3})}{\tau^3}\gap({P})^{-5}.		\end{aligned}
		\end{equation*} 
		This proves the lemma.
	\end{proof}
	
	\smallskip
	
	We are now ready to prove Lemma~\ref{lem:width-preserved-gap}.

	\smallskip

	\begin{proof}[Proof of Lemma~\ref{lem:width-preserved-gap}]
		Recall that  $\mathbf{F}$  is  a random point of  $\cO(P)$,
		and  
		that   $E$ is  the event that all comparisons in
		\eqref{eq:width-added-comparisons} hold in $\mathbf{F}$.
		Note that
		conditioning $\mathbf{F}$ on $E$ makes it uniform on $\cO(\cQ)$.
		The ground set is unchanged, so the normalization
		\. $Z_x=(n+1)F_x$ \. is the same for both posets.
		In particular,
		\[
		h_{\cQ}(x) \quad = \quad \Eb[Z_x\mid E].
		\]
		For any square-integrable random variable $V$,
		Cauchy--Schwarz gives
		\begin{equation}\label{eq:vodya}
			\bigl|\Eb[V\mid E]-\Eb[V]\bigr|
			\quad = \quad
			\frac{|\operatorname{Cov}(V,\mathbf{1}_E)|}{\Pb[E]}
			\quad \leq \quad
			\sqrt{\operatorname{Var}(V)}\sqrt{\frac{\Pb[E^c]}{1-\Pb[E^c]}}.
		\end{equation}
		On the other hand, plugging  the lower bound for $\gap(P)$ in assumption \eqref{eq:gap-bound-1} into Lemma~\ref{lem:q-small}, we get  
		\begin{align*}
			\Pb[E^c]
			&\quad  < \quad
			\frac{10e(1+\sqrt{3})\,\tau^{12}}{
				(18432\cdot510210\sqrt{3})^5
			} 
			\quad \approx \quad      (0.26\times10^{-1048}) \. 
			\frac{\tau^2}{\tau^2+{4} }
			\quad < \quad  \frac{\tau^2}{\tau^2+{4} },
		\end{align*}
		which implies 
		\[\sqrt{\frac{\Pb[E^c]}{1-\Pb[E^c]}}  \  \leq \  \frac{\tau}{2}. \]
		Plugging this into  \eqref{eq:vodya} gives us
		\begin{equation}\label{eq:vodya-2}
			\bigl|\Eb[V\mid E]-\Eb[V]\bigr|
			\quad \leq \quad
			\frac{\tau}{2} \. \sqrt{\operatorname{Var}(V)}.
		\end{equation}

		Now let  \. $V:=Z_t-Z_s$ \. for $s\in I$ and $t\in J$.
		{Since $s\in I$ and $t\in J$, we have}
		\[
		{\Delta(s,t) \ = \ h_{{P}}(t)-h_{{P}}(s)
			\ \geq \ h_{{P}}(v)-h_{{P}}(u) \ = \ \gap({P}).}
		\]
		We then have
		\begin{align*}
			& h_{\cQ}(t)-h_{\cQ}(s) \quad = \quad  \Eb[V\mid E]
			\quad   \geq \quad \Delta(s,t) \ - \  \frac{\tau}{2} \sqrt{\operatorname{Var}(V)}  \\
			& \quad \geq \quad \Delta(s,t) -\frac{\sqrt{\gap({P})\Delta(s,t)}}{2}
			\quad \geq \quad
			\frac{\Delta(s,t)}{2}
			\quad \geq \quad  \frac{\gap({P})}{2},
		\end{align*}
		where the first inequality is due to  \eqref{eq:vodya-2},
		the second inequality is due to Lemma~\ref{lem:width-gap-estimates},
		the third and fourth inequalities are because $\Delta(s,t) \geq \gap({P})$.
		This completes the proof.
	\end{proof}

	\medskip
	\section{Final remarks}	\label{s:finrem}
	
	\subsection{}\label{ss:finrem-Rota}
	To give a quick summary, the proof of Theorem~\ref{thm:main} draws on 
	{log-concavity} arising from
	deep results in convex geometry~\cite{Sta-log-concave,KS84,BFT95},
	combinatorial counting arguments in the spirit of~\cite{CPP-quant},
	and Shepp's XYZ correlation inequality~\cite{She82}.
	These are combined with the geometry of order polytopes,
	the window estimates in~\cite{AK25a},
	probabilistic conditioning, and Haqi's recent proof of the
	Haqi--Kahn ideal inequality~\cite{Haq26}.
	
	It is perhaps somewhat surprising to see so many different tools and techniques
	from across the area come together in the proof of such a small numerical
	improvement to break the infinite barrier.
	This brings to mind the following quote by Gian-Carlo Rota, from one 
	of his last publications:
	
	\begin{center}\begin{minipage}{13.25cm}%
			{{\em ``The solution $[...]$ always comes as
					an unexpected application of theories that were previously developed without a
					specific purpose, theories whose effectiveness was at first thought to be highly
					questionable.''} \cite{Rota98}.}
	\end{minipage}\end{center}
	
	
	\subsection{}\label{ss:finrem-AK}

	The proof of Theorem~\ref{thm:AK} in \cite{AK25b} combines two key ideas: (1) every poset of large range
	must contain an element $x$ whose position $f(x)$ in a uniformly random
	linear extension has large variance, and (2) every such element in a poset
	of bounded width belongs to a relatively balanced pair.
	The first assertion is by far the harder to establish: an element
	with large range need not be one whose position has large variance.
	
	Nevertheless, such an element can be found in the union of a maximal
	pair of sets with no comparabilities between them. The proof relies
	on a reduction to a special class of posets of width two, followed
	by a careful analysis of their linear extensions. For the second
	assertion, repeated applications of the XYZ inequality~\cite{She82} show that
	an element in a poset of bounded width that belongs to no relatively
	balanced pair must have bounded variance.
	Both parts make use of various correlation and log-concavity inequalities.

	\subsection{}\label{ss:finrem-1323}
	While the $\frac{1}{3}$--$\frac23$ \. Conjecture may seem like an 
	ultimate result in the area, the \defn{Kahn--Saks Conjecture} \. $\de(P)\to \frac{1}{2}$ \. 
	as \. $w(P)\to \infty$ \. resolved in \cite{Air26}, is much more natural.  
	Indeed, it was shown by Aigner \cite{Aig85} that for width $2$ posets the equality \. 
	$\de(P)=\frac{1}{3}$ \. holds if and only if $P$ is an ordinal sum of $C_1 \cup C_2$'s
	and $C_1$'s, where $C_n$ denotes the chain of $n$ elements and $\cup$ denotes the disjoint union of posets.  
	See also \cite{Chen18,Sah21} for variations on this result.  
	In fact, one would want to prove the polynomial rate of convergence \. 
	$\de(P)\to \frac{1}{2}$ \. as the last three authors do in \cite{CPP-sort,CPP-Cat}.

	\subsection{}\label{ss:finrem-hope}
	We expect---and sincerely hope---that the bound established in 
	this paper will soon be surpassed.  While we would be astonished
	to see a complete resolution of the \. $\frac{1}{3}$--$\frac23$ \. Conjecture, 
	perhaps one could expect an explicit improvement by, say, \ts $0.001$, 
	rather than by the unspecified microscopic amount in Theorem~\ref{thm:main}.  Clearly, new tools and
	fresh ideas are needed to obtain any such bound, especially given how slow 
	the progress has been until this point.

	\vskip.6cm
	
	{\small
		
		\subsection*{Use of AI}  

		Generative AI models were used to help with proof details and optimizations in several technical lemmas, 
		mainly Lemmas~\ref{lem:balance-displacement} and~\ref{lem:width-preserved-gap}. 
		The first arXiv preprint  of the paper ({\tt arXiv:2609.30888v1}) included relatively rough proofs of these two lemmas based on these arguments.
		For the current paper ({\tt arXiv:2609.30888v2}), the authors obtained without use of these models a simpler and stronger form of
		Lemma~\ref{lem:balance-displacement} and refined the proof of
		Lemma~\ref{lem:width-preserved-gap}, improving its clarity and adding
		intuitive explanations.
		We also used these models for light editing, reference checking and various computations, many of which were redone by the authors. 
		As always, the authors take full responsibility for 
		the content of the paper in its entirety.

		\subsection*{Acknowledgements}
		We are grateful to Jeff Kahn, J\'{a}nos~Koml\'{o}s,  and Mike Saks for many helpful discussions and remarks on the
		subject.  
		Swee Hong Chan, Igor Pak  and Greta Panova were partially  supported
		by the National Science Foundation.
	}



	

	\medskip	
	\appendix
	
	\section{Proof of Lemma~\ref{lem:BFT-C}}\label{s:BFT-C}
	We reproduce the proof of Brightwell--Felsner--Trotter~\cite[\S6]{BFT95}, making the constant $67/242\approx 0.27686$  explicit for use in this paper.

	We prove Lemma~\ref{lem:BFT-C} by contradiction.
	Suppose to the contrary that
	\[
	\delta(P;x,y) \quad \leq \quad \frac{67}{242},
	\qquad
	\delta(P;y,z) \quad \leq \quad \frac{67}{242}.
	\]
	Set
	\[
	p \ := \ \Pb[ \cL(y) \. < \. \cL(x)],
	\qquad
	p' \ := \ \Pb[ \cL(z) \. < \. \cL(y)].
	\]
	Since $h(x)\leq h(y)\leq h(z)$, Lemma~\ref{lem:KL} gives
	\[
	1-p \quad \geq \quad e^{-1} \quad > \quad \frac{67}{242},
	\qquad
	1-p' \quad \geq \quad e^{-1} \quad > \quad \frac{67}{242}.
	\]
	The  assumption at the beginning of the proof therefore implies 
	\[  0 \ < \ p \ =  \ \delta(P;x,y) \ \leq \ \frac{67}{242}, \qquad   
	0 \ < \ p' \ = \ \delta(P;y,z) \ \leq \ \frac{67}{242}.
	\]

	Let $t,t'\in(0,1]$ be
	\[
	t \ := \ \frac{\Pb[\cL(x)=\cL(y)+1]}{p},
	\qquad
	t' \ := \ \frac{\Pb[\cL(y)=\cL(z)+1]}{p'}.
	\]
	By~\cite[Lem.~6.2]{BFT95},
	\begin{equation*}
		\Pb[\cL(x)-\cL(y) \geq 2] \  + \   \Pb[\cL(y)-\cL(z) \geq 2]  \ \geq \ \frac{1}{11},
	\end{equation*}
	which in our notation is equivalent to
	\begin{equation}\label{eq:BFT-C-adjacency}
		p(1-t)+p'(1-t') \quad \geq \quad \frac{1}{11}.
	\end{equation}

	We recall the definition of the packed height function from~\cite[Section~2]{BFT95}.
	Fix \. $q \in (0,\tfrac{67}{242}]$ \. and \. $u\in (0,1]$\..
	A pair of nonnegative sequences $(a_i)_{i\geq1}$ and $(b_i)_{i\geq1}$ is called \defnb{packed} with parameters $(q,u)$ if, for some integer $k\geq1$,
	\[
	b_i \ = \ qu(1-u)^{i-1} \quad (i\geq1),
	\qquad
	a_i \ = \ qu(1+u)^{i-1} \quad (1\leq i\leq k),
	\]
	and
	\[
	a_{k+1}+a_{k+2} \ = \ 1-q-\sum_{i=1}^{k}a_i,
	\qquad
	a_i \ = \ 0 \quad (i>k+2),
	\]
	with one of the following two conditions:
	\begin{enumerate}[label=\textnormal{(\roman*)}]
		\item $k\geq2$, \. $a_{k+2}=0$, \. and
		\[
		\frac{u}{1+u} \quad \leq \quad \frac{a_{k+1}}{a_k}
		\quad \leq \quad 1.
		\]
		\item $a_{k+1}=a_k+a_{k+2}$ \. and
		\[
		1 \quad \leq \quad \frac{a_{k+1}}{a_k}
		\quad \leq \quad 1+u.
		\]
	\end{enumerate}
	Note that these sequences satisfy  \. $\sum_{i\geq1}b_i=q$ \. and \. $b_1/q=u$.
	For each $(q,u)$ in this range, it was shown in \cite[Sec.~3]{KS84} that there is a unique pair of sequences of this form.
	The \defnb{packed height function} is defined by
	\[
	H(q,u) \quad := \quad \sum_{i\geq1} i(a_i-b_i).
	\]
	The formulas in~\cite[(2.9)--(2.10)]{BFT95} give
	\begin{equation}\label{eq:H-cases}
		H(q,u) \ = \
		\begin{cases}
			\displaystyle
			k+1-\frac{q(1+u)^{k+1}}{u},
			& \text{in case (i)}, \\[8pt]
			\displaystyle
			k+\frac{3}{2}
			-\frac{q(1+u)^{k-1}(4u^2+5u+2)}{2u},
			& \text{in case (ii)}.
		\end{cases}
	\end{equation}

	\smallskip

	\begin{claim}\label{lem:Height}
		For  any $q \in (0,\tfrac{67}{242}]$ and $u \in (0,1]$,
		\begin{equation}\label{eq:H(q,1)}
			H(q,u) \quad \geq \quad H(q,1)
			\quad \geq \quad \frac{5-11q}{2}.    
		\end{equation}
	\end{claim}
	
	\smallskip
	
	\begin{proof}
		{The first inequality follows from}
		the monotonicity lemma in~\cite[Lem.~2.3]{BFT95} (note that  we invoke $\tfrac{67}{242}<1-1/\sqrt{2} \approx 0.293$ to use this lemma).
		We now prove the second inequality.
		First, for  $q\geq 1/5$, the packed sequence at $u=1$ is in case~(ii) with $k=1$. Thus the formula for $H$ gives
		\begin{equation}\label{eq:Hyatt-evil}
			H(q,1)
			\quad = \quad
			1+\frac{3}{2}-\frac{q(4+5+2)}{2}
			\quad = \quad
			\frac{5-11q}{2}.
		\end{equation}

		Now suppose that $0<q<1/5$.
		At $u=1$, the definition of a packed sequence gives
		\[
		b_1=q,\qquad b_i=0 \quad (i\geq2),
		\qquad a_1=q,\qquad a_2=2q.
		\]
		Since $\sum_{i\geq1}a_i=1-q$, the remaining positive mass is
		\[
		\sum_{i\geq3}a_i \quad = \quad 1-4q.
		\]
		Every term in this remaining mass has index at least three. It thus follows from the two equations above that
		\begin{align*}
			&H(q,1)
			\ = \ \sum_{i\geq1} i (a_i-b_i) \ = \   \sum_{i\geq1} i a_i-q  \ \geq \ q+4q+3(1-4q)-q \\
			& \quad \ = \  3-8q \ \geq \ \frac{5-11q}{2},
		\end{align*}
		where the last inequality uses the assumption that $q < 1/5$. This completes the proof.
	\end{proof}

	\smallskip

	{The packing lemma~\cite[Lem.~2.2]{BFT95} gives}
	\begin{equation}\label{eq:packing}
		2 \quad \geq \quad h(z)-h(x)
		\quad \geq \quad H(p,t)+H(p',t').
	\end{equation}
	It then follows from \eqref{eq:H(q,1)} that
	\[   2 \ \geq \ 5-  \frac{11}{2} (p+p')  \quad \text{ and thus } \quad   p+p' \  \geq \ \frac{6}{11},\]
	which in turn implies 
	\[
	\frac{65}{242}	\ = \  \frac{6}{11}-\frac{67}{242} \quad \leq \quad  p,p' \quad  \leq \quad \frac{67}{242}.
	\]
	In particular, we have  $p,p' \geq 1/5$,  so it follows from \eqref{eq:Hyatt-evil}
	that 	
	\begin{equation}\label{eq:Hyatt-really-evil}
		H(p,1) \ = \ \frac{5-11p}{2},
		\qquad
		H(p',1) \ = \ \frac{5-11p'}{2}.
	\end{equation}
	For $q\in\{p,p'\}$, let $v_q$ be the positive solution of
	\[
	(1+v_q)(1+2v_q) \quad = \quad \frac{1}{q}.
	\]
	Since   $q \in [\tfrac{65}{242}, \tfrac{67}{242}]$, we have  \. $3/5<v_q<1$\..

	Suppose first that $t\leq v_p$ or $t'\leq v_{p'}$, say $t\leq v_p$.
	Writing $v:=v_p$,
	we then have
	\[	p \ = \ \frac{1}{(1+v)(1+2v)}.
	\]
	This implies
	\[
	H(p,v) \ = \ 3-\frac{p(1+v)^3}{v}, \qquad  H(p,1) \ = \  \frac{5-11p}{2},
	\]
	{where the first equality holds at the boundary of case~(i) with $k=2$},
	and the second is because of \eqref{eq:Hyatt-really-evil}.
	It then follows from
	{the monotonicity lemma~\cite[Lem.~2.3]{BFT95}}  that
	\begin{equation}\label{eq:Hyatt-1}
		H(p,t)-H(p,1)
		\quad \geq \quad H(p,v)-H(p,1)
		\quad = \quad
		\frac{-3v^2+6v-2}{2v(1+v)(1+2v)}
		\quad \geq \quad \frac{1}{12},
	\end{equation}
	where in the last inequality we use the fact that $1/2 \leq v \leq 1$.
	Consequently,
	\begin{align*}
		&h(z)-h(x)
		\quad \geq \quad H(p,t)+H(p',t') \\
		&\quad  \geq \quad  H(p,1)+H(p',1)+\frac{1}{12}  \quad  = \quad \frac{5-11p}{2}+\frac{5-11p'}{2}+\frac{1}{12}\\
		&  \quad  \geq \quad  5-11\cdot \frac{67}{242}+\frac{1}{12} \quad > \quad 2,
	\end{align*}
	where the first inequality uses \eqref{eq:packing}, the second inequality uses \eqref{eq:Hyatt-1},
	and the equality uses \eqref{eq:Hyatt-really-evil}.
	This contradicts  the assumption that $(x,y,z)$ is a BFT triple.
	We may therefore assume that \. $t>v_p$ \. and \. $t'>v_{p'}$\..
	\smallskip

	\begin{claim}Assume  that \. $t>v_p$ \. and \. $t'>v_{p'}$\..
		Then
		\begin{equation}\label{eq:BFT-C-gain}
			H(p,t)-H(p,1) \quad \geq \quad \frac{1}{2}p(1-t).
		\end{equation}
		Moreover, the inequality is strict whenever $t<1$.
		The analogous statements hold for $p',t'$.	\end{claim}
	
	\smallskip
	
	\begin{proof}
		The case $t=1$ is immediate, so we assume $t<1$.
		Also note that $t>v_p>3/5$.
		We start with the case  \. $1/p\leq1+2t+2t^2$\..
		This hypothesis places $(p,t)$ in packed case~(ii) with $k=1$.
		Hence
		\[
		H(p,t) \ = \ \frac{5}{2}-\frac{p(4t^2+5t+2)}{2t}, \quad  H(p,1) \ = \ \frac{5-11p}{2},
		\]
		which in turn implies
		\begin{align*}
			H(p,t)-H(p,1)
			\ = \ p(1-t)\left(2-\frac{1}{t}\right).
		\end{align*}
		Moreover, the inequalities \. $1/p\leq1+2t+2t^2$ \. and $p\leq \frac{67}{242}$ imply that \. $(2-1/t) > 1/2$\..
		Thus \eqref{eq:BFT-C-gain} holds in this case.
		
		We now deal with the case  \. $1/p>1+2t+2t^2$\..
		Note that this implies
		\[ 1+2t+2t^2 \quad < \quad \frac{1}{p} \quad < \quad (1+t)(1+2t), \]
		where the right inequality is because  of the assumption that \. $t>v_p$\..
		These inequalities then put us in packed case~(i) with $k=2$.
		The formula for case~(i) in~\eqref{eq:H-cases} therefore gives
		\[
		H(p,t) \quad = \quad 3-\frac{p(1+t)^3}{t}.
		\]
		Set
		\[
		g(t) \ := \ \frac{2}{t}-4+5t+2t^2.
		\]
		A direct calculation gives
		\[
		H(p,t)-H(p,1)-\frac{1}{2}p(1-t)
		\quad = \quad \frac{1-p \. g(t)}{2}.
		\]
		If \. $t\geq2/3$\., then
		\[
		1+2t+2t^2-g(t)
		\quad = \quad \frac{(1-t)(3t-2)}{t}
		\quad \geq \quad 0,
		\]
		so $g(t)< 1/p$.
		If $3/5<t<2/3$, bounding the terms of $g(t)$ separately gives
		\[
		g(t) \quad < \quad \frac{32}{9}
		\quad < \quad \frac{242}{67}
		\quad \leq \quad \frac{1}{p}.
		\]
		This proves~\eqref{eq:BFT-C-gain} in both scenarios.
		The same argument applies to $p',t'$.
	\end{proof}
	\smallskip

	We now have
	\begin{align*}
		h(z)-h(x)
		&\ > \ H(p,1)+H(p',1)
		+\frac{1}{2}\bigl(p(1-t)+p'(1-t')\bigr) \\
		&\ \geq  \ 5-\frac{11}{2}(p+p')+\frac{1}{22} \\
		&\ \geq \ 5-11\cdot\frac{67}{242} +\frac{1}{22}
		\ = \ 2,
	\end{align*}
	where the first inequality is due to \eqref{eq:packing} {and}~\eqref{eq:BFT-C-gain} (note that we can use the strict version of \eqref{eq:BFT-C-gain} because at least one of $t,t'$ is strictly less than $1$ by \eqref{eq:BFT-C-adjacency}),
	and the second inequality is due to \eqref{eq:Hyatt-really-evil} and \eqref{eq:BFT-C-adjacency}.
	This contradicts the defining inequality for a BFT triple and completes the proof.
	\qed
	
	\medskip


{\footnotesize
		
		\begin{thebibliography}{abcdefghi}
			
			\bibitem[AD78]{AD78}
			Rudolf~Ahlswede and David~E.~Daykin,
			An inequality for the weights of two families of sets, their unions and intersections,
			\emph{Z.~Wahrsch.\ Verw.\ Gebiete}~\textbf{43} (1978), 183--185.
			
			\bibitem[Aig85]{Aig85}
			Martin~Aigner, A note on merging, \emph{Order} \textbf{2} (1985), 257--264. 
			
			\bibitem[Air26]{Air26}
			Max~Aires, Proof of the Kahn--Saks conjecture, preprint (2026), 25~pp.; \ts {\tt arXiv:2609.30895}.  		
			
			\bibitem[AK25a]{AK25a}
			Max~Aires and Jeff~Kahn,
			Balancing extensions in posets of large width,
			preprint (2025), 24~pp.; \ts 
			{\tt arXiv:2509.11549}.
			
			
			
			
			
			\bibitem[AK25b]{AK25b}
			Max~Aires and Jeff~Kahn,
			Variance vs.\ range for linear extensions, and balancing extensions in posets of bounded width,
			preprint (2025), 11~pp.; \ts {\tt arXiv:2510.26134v1}.
			
			
			\bibitem[AS16]{AS16}
			Noga~Alon and Joel~H.~Spencer, \emph{The probabilistic method} (Fourth ed.),
			John Wiley, Hoboken, NJ, 2016, 375~pp.
			
			\bibitem[BLL98]{BLL98}
			Fran\c{c}ois~Bergeron, Gilbert~Labelle and Pierre~Leroux,  
			\emph{Combinatorial species and tree-like structures}, 
			Cambridge Univ.\ Press, Cambridge, UK, 1998, 457~pp.
			
			
			\bibitem[Bri88]{Bri88}
			Graham~Brightwell,
			\emph{Linear extensions of infinite posets},
			\emph{Discrete Math.} \textbf{70} (1988), 113--136.
			
			\bibitem[Bri89]{Bri89}
			Graham~Brightwell,
			Semiorders and the {$\frac 13$}--{$\frac 23$} conjecture,
			\emph{Order}~\textbf{5} (1989), 369--380.
			
			\bibitem[Bri99]{Bri99}
			Graham~Brightwell,  Balanced pairs in partial orders,
			\emph{Discrete Math.}~\textbf{201} (1999), 25--52.
			
			\bibitem[BFT95]{BFT95}
			Graham~Brightwell, Stefan~Felsner and William~T.~Trotter,
			Balancing pairs and the cross product conjecture,
			\emph{Order}~\textbf{12} (1995), 327--349.
			
			
			\bibitem[BW92]{BW92}
			Graham~Brightwell and Colin~Wright,
			The {$1/3\text{--}2/3$}  conjecture for $5$-thin posets,
			\emph{SIAM J.\ Discrete Math.}~\textbf{5} (1992), 467--474.
			
			\bibitem[BZ88]{BZ-book}
			Yuri~D.~Burago and  Victor~A.~Zalgaller,
			\emph{Geometric inequalities},
			Springer, Berlin, 1988, 331~pp.
			
			\bibitem[CP25]{CP-LE}
			Swee~Hong~Chan and Igor~Pak, Linear extensions of finite posets,
			\emph{EMS Surveys in Mathematical Sciences},
			published online (2025), 56~pp.
			
			\bibitem[CPP21a]{CPP-sort}
			Swee~Hong~Chan, Igor~Pak and Greta~Panova,
			Sorting probability for large Young diagrams, \emph{Discrete Anal.}~\textbf{2021} (2021),
			Paper No.~24, 57~pp.
			
			
			\bibitem[CPP21b]{CPP-Cat}
			Swee~Hong~Chan, Igor~Pak and Greta~Panova, Sorting probability of Catalan posets,
			\emph{Advances Applied Math.}~\textbf{129} (2021), Paper No.~102221, 13~pp.
			
			\bibitem[CPP25]{CPP-quant}
			Swee~Hong~Chan, Igor~Pak and Greta~Panova,
			On the cross-product conjecture for the number of linear extensions,
			\emph{Canadian Jour.\ Math.} \textbf{77} (2025), 535--562.
			
			\bibitem[Chen18]{Chen18}
			Evan~Chen,  A family of partially ordered sets with small balance constant, 
			\emph{Electron.\ J.\ Combin.} \textbf{25} (2018), no.~4, Paper No.~4.43, 13~pp.
			
			
			
			\bibitem[FKG71]{FKG71}
			Cornelius~M.~Fortuin, Pieter~W.~Kasteleyn and Jean~Ginibre,
			Correlation inequalities on some partially ordered sets,
			\emph{Comm.\ Math.\ Phys.}~\textbf{22} (1971), 89--103.
			
			
			\bibitem[Fre75]{Fre75}
			Michael~L.~Fredman,
			How good is the information theory bound in sorting?,
			\emph{Theoret.\ Comput.\ Sci.}~\textbf{1} (1976), 355--361.
			
			
			\bibitem[Fri93]{Fri93}
			Joel~Friedman,
			A note on poset geometries,
			\emph{SIAM J.\ Comput.}~\textbf{22} (1993), 72--78.
			
			\bibitem[GHP87]{GHP87}
			Bernhard~Ganter, Gerhart~H\"{a}fner and Werner~Poguntke,
			On linear extensions of ordered sets with a symmetry,
			\emph{Discrete Math.}~\textbf{63} (1987), 153--156.
			
			\bibitem[Gru60]{Gru60}
			Branko~Gr\"unbaum,
			Partitions of mass-distributions and of convex bodies by hyperplanes,
			\emph{Pacific J.~Math.}~\textbf{10} (1960), 1257--1261.
			
			\bibitem[Gup26]{Gup26}
			Anish Gupta,
			Balance constants, majority cycles,
			and the Gold Partition Conjecture
			through fourteen elements, preprint (2026), 16~pp.; \ts 
			{\tt arXiv:2607.23926}.
			
			
			\bibitem[Haq26]{Haq26}
			Alireza~Haqi,
			On the gap of finite posets, preprint (2026), 25~pp.; \ts 
			{\tt arXiv:2608.12678}.
			
			\bibitem[KL91]{KL90}
			Jeff~Kahn and Nathan~Linial, Balancing extensions via Brunn--Minkowski,
			\emph{Combinatorica}~\textbf{11} (1991), 363--368.
			
			\bibitem[KS84]{KS84}
			Jeff~Kahn and Michael~Saks,
			Balancing poset extensions,
			\emph{Order}~\textbf{1} (1984), 113--126.
			
			
			
			\bibitem[Kis68]{Kis68}
			Sergey~S.~Kislitsyn, A finite partially ordered set and its
			corresponding set of permutations,
			\emph{Math.\ Notes}~\textbf{4} (1968), 798--801.
			
			\bibitem[Kom90]{Kom90}
			J\'{a}nos~Koml\'{o}s, A strange pigeonhole principle,
			\emph{Order}~\textbf{7} (1990), 107--113.
			
			
			\bibitem[LV07]{LV07}
			L\'aszl\'o~Lov\'asz and Santosh~Vempala,
			The geometry of logconcave functions and sampling algorithms,
			\emph{Random Structures Algorithms}~\textbf{30} (2007), 307--358.
			
			
			
			\bibitem[Lin84]{Lin84}
			Nathan~Linial, The information-theoretic bound is good for merging,
			\emph{SIAM J.\ Comput.}~\textbf{13} (1984), 795--801.
			
			\bibitem[Mat02]{Mat02}
			Ji\v{r}\'{\i}~Matou\v{s}ek, \emph{Lectures on discrete geometry},
			Springer, New York, 2002, 481~pp.
			
			
			\bibitem[MPP18]{MPP-phi}
			Alejandro~H.~Morales, Igor~Pak and Greta~Panova,
			Why is \ts $\pi < 2\ts\phi$? \ts
			\emph{Amer.\ Math.\ Monthly}~\textbf{125} (2018), 715--723.
			
			\bibitem[OS18]{OS18}
			Emily~J.~Olson and Bruce~E.~Sagan,
			On the \ts {$\frac 13$}--{$\frac 23$} \ts  conjecture,
			\emph{Order}~\textbf{35} (2018), 581--596.
			
			\bibitem[Pec08]{Pec08}
			Marcin~Peczarski,
			The gold partition conjecture for $6$-thin posets,
			\emph{Order}~\textbf{25} (2008), 91--103.
			
			
			\bibitem[Rota98]{Rota98}
			G.-C.~Rota, Foreword to \cite{BLL98}.  
			
			\bibitem[Sah21]{Sah21} Ashwin~Sah, 
			Improving the {$\frac{1}{3}$}-{$\frac{2}{3}$} conjecture for
			width two posets, \emph{Combinatorica} \textbf{41} (2021), 99--126.
			
			
			\bibitem[She82]{She82}
			Lawrence~A.~Shepp,
			The XYZ conjecture and the FKG inequality,
			\emph{Ann.\ Probab.}~\textbf{10} (1982), 824--827.
			
			
			
			\bibitem[Sta81]{Sta-AF}
			Richard~P.~Stanley,
			Two combinatorial applications of the Aleksandrov--Fenchel inequalities,
			\emph{J.\ Combin.\ Theory, Ser.~A} \textbf{31}  (1981), 56--65.
			
			\bibitem[Sta86]{Sta-two}
			Richard~P.~Stanley, Two poset polytopes,
			\emph{Discrete Comput.\ Geom.}~\textbf{1} (1986), no.~1, 9--23.
			
			\bibitem[Sta89]{Sta-log-concave}
			Richard~P.~Stanley,
			Log-concave and unimodal sequences in algebra, combinatorics, and geometry,
			in  \emph{Graph theory and its applications}, New York Acad.\ Sci.,
			New York, 1989, 500--535.
			
			
			
			\bibitem[Zag12]{Zag12}
			Imed~Zaguia,
			The \ts {$\frac 13$}--{$\frac 23$} \ts conjecture for $N$-free ordered sets,
			\emph{Electron.\ J.\ Combin.}~\textbf{19} (2012), no.~2, Paper~29, 5~pp.
			
			
			\bibitem[Zag19]{Zag19}
			Imed~Zaguia,
			The \ts {$\frac 13$}--{$\frac 23$} \ts conjecture for ordered sets whose cover
			graph is a forest, \emph{Order}~\textbf{36} (2019), 335--347.
			
			
			
			
			
		\end{thebibliography}
	}
\end{document}